\documentclass{amsart}
\usepackage[utf8]{inputenc}
\usepackage{fullpage}
\usepackage{tikz}
\usepackage{HZ}
\usepackage{amssymb}
\usepackage{caption}
\usepackage{xcolor}
\usepackage{lscape}
\usepackage{CJKutf8}
\usepackage{tabularray}
\usepackage{subfig}

\usetikzlibrary{decorations}
\usetikzlibrary{decorations.pathmorphing}
\usetikzlibrary{ decorations.markings}
\usetikzlibrary{calc}

\def\gridsize{60}
\def\topEnd{26}
\def\bottomEnd{-36}

\newtheorem{theorem}{Theorem}
\newtheorem{proposition}[theorem]{Proposition}
\newtheorem{corollary}[theorem]{Corollary}
\newtheorem{lemma}[theorem]{Lemma}
\theoremstyle{definition}
\newtheorem{definition}[theorem]{Definition}

\theoremstyle{remark}

\numberwithin{equation}{section}
\tikzset{
  mountain/.style = {thick},
  valley/.style = {thick, densely dashed},
  mountain_opt/.style = {thick, color=blue!70!black},
  valley_opt/.style = {thick, densely dashed, color=blue!70!black},
  opt/.style = {thick, color=green!70!black}
}

\title{Flat origami is Turing complete}

\author[Hull]{Thomas C. Hull}
\address{Franklin \& Marshall College}
\curraddr{}
\email{thomas.hull@fandm.edu}
\thanks{}

\author[Zakharevich]{Inna Zakharevich}
\address{Cornell University}
\curraddr{}
\email{zakh@math.cornell.edu}
\thanks{}

\begin{document}

\maketitle

\begin{abstract}
\textit{Flat origami} refers to the folding of flat, zero-curvature paper such that the finished object lies in a plane. Mathematically, flat origami consists of a continuous, piecewise isometric map $f:P\subseteq\mathbb{R}^2\to\mathbb{R}^2$ along with a layer ordering $\lambda_f:P\times P\to \{-1,1\}$ that tracks which points of $P$ are above/below others when folded. The set of crease lines that a flat origami makes (i.e., the set on which the mapping $f$ is non-differentiable) is called its \textit{crease pattern}. Flat origami mappings and their layer orderings can possess surprisingly intricate structure. For instance, determining whether or not a given straight-line planar graph drawn on $P$ is the crease pattern for some flat origami has been shown to be an NP-complete problem, and this result from 1996 led to numerous explorations in computational aspects of flat origami. In this paper we prove that flat origami, when viewed as a computational device, is Turing complete, or more specifically P-complete. We do this by showing that flat origami crease patterns with \textit{optional creases} (creases that might be folded or remain unfolded depending on constraints imposed by other creases or inputs) can be constructed to simulate Rule 110, a one-dimensional cellular automaton that was proven to be Turing complete by Matthew Cook in 2004. 
\end{abstract}

\tableofcontents

\section{Introduction}

Origami, the art of paper folding, has lately been a source of inspiration for
applications in mechanical engineering \cite{Fulton21,YOrigami}, materials
science \cite{Silverberg1,Liu}, and architecture \cite{Maleczek2018}.  Helping
this interest has been the rise of \textit{computational origami}, which studies
computational questions that emerge from the folding of paper, as a field in
computational and combinatorial geometry \cite{GFALOP}. Of particular interest
has been \textit{flat origami}, where a two-dimensional sheet of paper, or all
of $\mathbb{R}^2$, is folded into a flat object, back into the plane without
stretching, tearing, or self-intersecting the paper. For example, in 1996 Bern
and Hayes proved that the decidability question of whether a given crease
pattern can fold flat is NP-hard \cite{Bern}. However, because of the difficulty
in rigorously modeling flat origami, a hole in their proof remained undetected
for 20 years until Akitaya et al. repaired and strengthened their proof in 2016
\cite{boxpleat}.

The fact that flat origami is so computationally difficult may seem strange and
contradictory, but this is because it involves a complicated interaction between
local and global phenomena.  For example, Kawasaki's Theorem (see Theorem~\ref{thm:Kawasaki} below) states
that a set of creases which all meet at a vertex can fold flat if and only
if the alternating sum of its sector angles is zero.  However, this condition is not
sufficient for a multiple-vertex crease pattern or even crease patterns with non-intersecting lines to fold flat; for this, it is necessary
to consider global properties.  For instance, if we make two parallel creases on
a sheet of paper, they can fold flat in the same direction only if the distance
between the creases is large enough to fit the two flaps.  Thus even though each
individual crease can flat-fold, the two creases can interfere with one another.
With more complicated crease patterns these interactions become more and more
complicated, with a complete analysis requiring an understanding of the order of
different layers of the paper (see Section~\ref{sec:prelim}).  Thus the global structure of the crease pattern
massively increases the complexity of flat origami.

In a similar vein, we can consider finite cellular automata: ``machines'' with
cells, each of which can have only finitely many states, and where cells shift states
based on the states of their neighbors.  If a finite cellular automaton only has
finitely many cells then it's a finite state machine, with very little
computational power.  But cellular automata with infinitely many cells, even
very restricted ones where each cell only sees two neighbors, are extremely
computationally powerful.  Given that flat-folding origami is much more
complex than the cellular case, even for only finitely many folds, the question
arises: How complex is flat-folding origami when it is placed in an infinite
sheet, with infinitely many creases allowed?

A problem containing an infinite amount of information can be arbitrarily
complex.  We therefore decide to limit the problem. We start with an \textit{origami
tessellation} (i.e. a periodic crease pattern that is possible to fold flat) on an infinite sheet of paper.
Then we allow ourselves to modify only finitely many of the creases in the
tessellation and seek to determine the complexity of the problem: Can this new crease configuration
fold flat?  The idea is that our tessellation would model the behavior of a
cell in a one-dimensional cellular automaton, with the cells going across the
paper, left-to-right, and the ``result of the computation'' going down the paper.  If the
computation terminates, the ``modification'' would be constrained to a finite
portion of the paper; if it did not then the ``necessary modification'' for the
pattern to flat-fold would go off to infinity.

% More formally, suppose that $f: \mathbb{N}^k \rightarrow \mathbb{N}$ is a computable
% function, i.e. there is a Turing machine with the properties that (REFERENCE!!)
% \begin{itemize}
% \item (Currently off of WIkipedia) If $f(\mathbf{x})$ is defined then the
%   machine will terminate on the input $\mathbf{x}$ with the output
%   $f(\mathbf{x})$.
% \item If $f(\mathbf{x})$ is not defined the program never terminates on the
%   input $\mathbf{x}$.
% \end{itemize}
% Notice that for any computable function $f$ there is a Turing machine that
% terminates with \emph{empty} output on the input $\mathbf{x}$ whenever
% $f(\mathbf{x})$ is defined, and never terminates if $f(\mathbf{x})$ if it is
% not.  There exists a linear cellular automaton (DEFINE THIS EARLIER, REFRENCE
% COOK) that can model any Turing machine, and in particular can model this
% modification.  On any input, it will either eventually become periodic (with the
% ``empty input'') or it will not; thus the question of whether the Turing machine
% terminates is equivalent to the question of whether the ``deviation from the
% periodicity'' is contained in a finite region or not.  This is exactly the form
% of the flat-folding origami question that we have stated.

In this paper we consider a very simple example of a cellular automaton called
Rule 110 (see Section~\ref{sec:Rule110} for a more detailed definition).  This
cellular automaton consists of an infinite line of cells, each of which can be
``on'' or ``off.''  At each click of the computation, a cell changes its state
based on its current state and the states of its two neighbors.  Thus the
``computational history'' of the entire process can be visualized as a
two-dimensional grid of squares, with each row being a state of the computation
and each column recording the sequence of states that a cell went through.  In
\cite{Cook}, Cook showed that Rule 110 is Turing-complete, and not long after Neary 
and Woods \cite[Theorem 1]{Woods} showed that Rule 110 is
logspace-complete for \textbf{P}.  In this paper we prove the following theorem:
\begin{theorem}
  It is possible to construct an origami tessellating cell which can model
  Rule 110: i.e. the output of the cell is determined from the three
  inputs via Rule 110.
\end{theorem}
In particular, we can take the entire history of a computation using Rule 110
and model it using a tessellating cell in origami. See Section~\ref{sec:tess} for a more in-depth discussion. As a corollary, we get the
following statement which is more specific than the claim made by this paper's title of origami being Turing complete:
\begin{corollary}
  Tessellating flat-foldable origami is logspace-complete for \textbf{P}.
\end{corollary}
Our approach is to make use of \textit{optional creases} in our crease patterns
to help encode Boolean variables and design logic gates, an approach that has
been used in prior work to explore the complexity of origami \cite{rigidcom}.
In Section~\ref{sec:prelim} we will formally define our model of flat-foldable
origami, define Rule 110, and establish conventions in our approach.  In
Section~\ref{sec:gates} we will define and prove the correctness of the origami
gadgets we will use to transmit Boolean signals and simulate logic gates.  In
Section~\ref{sec:110} we will put our gadgets together to simulate cellular
automata, in particular Rule 110.

\subsection*{Acknowledgements}  The first author is supported in part by NSF DMS-1906202
and DMS-2347000. The second author is supported in part by NSF DMS-1846767.  The
authors would also like to thank Damien Woods for helpful comments and the reference to his paper \cite{Woods}.

\section{Conventions and preliminaries}\label{sec:prelim}

What it means to ``fold a piece of a paper'' or for a crease pattern to ``fold
flat'' is fairly intuitively obvious.  However, as with many intuitively obvious
things, it turns out to be fairly difficult to write rigorously.  This section
lays out the rigorous definitions and background necessary for the techinical
analysis in the paper, but it is not necessary for a first (or even a third)
understanding, and many readers can freely skip it.

\subsection{Background}\label{ssec:back}

We follow a model and terminology for planar, two-dimensional flat origami as presented in \cite{boxpleat} and \cite{GFALOP}.\footnote{The flat-folding of manifolds in general dimension is also possible and follows many of the properties of the flat, 2D case presented here.  See \cite{Robertson} and \cite[Chapter 10]{origametry} for more details.} A flat-folded piece of paper may be modeled using two structures: an isometric folding map and a layer ordering.  An \textit{isometric folding map} is a continuous, piecewise isometry $f:P\subseteq\mathbb{R}^2\to\mathbb{R}^2$ where $P$ is closed. The \textit{crease pattern} of $f$, denoted $X_f$, is the set of points on $P$ on which $f$ is non-differentiable, union with the boundary of $P$. One can prove \cite{origametry, Robertson} that 
\begin{itemize}
\item $X_f$ is a plane graph on $P$ whose interior edges, which we call \textit{creases}, are straight line segments,
\item every interior vertex of $X_f$ has even degree,
\item the faces defined by the embedding of $X_f$ on $P$ are 2-colorable, where one color class is made of regions of $P$ whose orientation are preserved under $f$ and the other color class faces are orientation-reversed under $f$, and
\item around each interior vertex $v$ of $X_f$ the alternating sum of the sector angles between the creases at $v$, say going in order counterclockwise, equals zero (this is called \textit{Kawasaki's Theorem}).
\end{itemize}

We will use Kawasaki's Theorem throughout our proofs, and so we formalize what it states for a single vertex in a crease pattern: 

\begin{theorem}[Kawasaki's Theorem {\cite[Theorem 5.37]{origametry}}]\label{thm:Kawasaki}
A collection of line segments or rays that share a common endpoint $v\in\mathbb{R}^2$ and whose consecutive sector angles are $\alpha_1,\ldots, \alpha_{2n}$ will be flat-foldable (meaning they are part of a crease pattern $X_f$ for some isometric folding map $f$) if and only if $\sum (-1)^k\alpha_k = 0$. Since our crease patterns exist in a flat plane, this is equivalent to 
$$\alpha_1+\alpha_3+\cdots+\alpha_{2n-1} = \alpha_2+\alpha_4+\cdots+\alpha_{2n} = \pi.$$
\end{theorem}

\begin{figure}
\centering
 \begin{tikzpicture}[scale=2]
 \draw (0,0) ellipse (.7 and .35);
 \fill[white] (0,.2) ellipse (.7 and .35);
 \draw (0,.2) ellipse (.7 and .35);
 \node at (-1,.3) {$f(U_1)$};
 \node at (-1,-.1) {$f(U_2)$};

\draw (1.2,-.45) .. controls (1.9,-.65) and (3.1,.15) .. (2.6,.4);
\fill[white] (2,.1) ellipse (.7 and .35);
\draw (2,.1) ellipse (.7 and .35);
\fill[white] (1.2,-.3) .. controls (1.9,-.5) and (3.1,.3) .. (2.5,.5);
\draw (1.2,-.3) .. controls (1.9,-.5) and (3.1,.3) .. (2.5,.5);
\draw (1.2,-.3) to (2.5,.5);
\draw (1.2,-.3) to (1.2,-.45);
%\draw (1.4,-.45) arc (270:90:.075);
 \node at (1.07,.2) {$f(U_2)$};
 \node at (2.1,0) {$f(U_1)$};
 \node at (2.1,-.5) {$f(U_3)$};
 \node at (1.85,.2) {$e$};

\draw (3.2,-.6) .. controls (3.9,-.8) and (5.3,0) .. (4.6,.25); 
\fill[white] (3.2,-.4) .. controls (3.9,-.6) and (5.3,.2) .. (4.6,.45);
\draw (3.2,-.4) .. controls (3.9,-.6) and (5.3,.2) .. (4.6,.45);
\fill[white] (3.2,-.2) .. controls (3.9,-.4) and (5.3,.4) .. (4.6,.65);
\draw (3.2,-.2) .. controls (3.9,-.4) and (5.3,.4) .. (4.6,.65);
\draw (3.2,-.2) to (4.6,.45);
\fill[white] (3.2,0) .. controls (3.9,-.2) and (5.3,.6) .. (4.6,.8);
\draw (3.2,0) .. controls (3.9,-.2) and (5.3,.6) .. (4.6,.8);
\draw (3.2,0) to (4.6,.8);

\draw (3.2,-.4) .. controls (3.1,-.4) and (3.1,0) .. (3.2,0);
\draw (3.2,-.6) .. controls (3.1,-.6) and (3.1,-.2) .. (3.2,-.2);
\node at (5.05,.7) {$f(U_1)$};
\node at (5.05,.45) {$f(U_2)$};
\node at (5.05,.2) {$f(V_1)$};
\node at (5.05,.-.05) {$f(V_2)$};
\node at (3.03,-.2) {$e_1$};
\node at (3.03,-.45) {$e_2$}; 

\node at (-1.1,.7) {(a)};
\node at (1,.7) {(b)};
\node at (3,.7) {(c)};
\end{tikzpicture}
\caption{(a) The tortilla-tortilla condition being satisfied.  (b) The taco-tortilla condition \textit{not} being satisfied. (c) The taco-taco condition \textit{not} being satisfied.}
\label{fig:tacos}
\end{figure}

Modeling flat-folded origami also requires the concepts of layer ordering and mountain-valley creases, which require additional structure be added to an isometric folding map.  First, we introduce some terminology. A simply connected subset of $U\subset P$ is called \textit{uncreased} under $f$ if $f$ restricted to $U$ is injective. Two simply connected subsets $U_1,U_2\subset P$ \textit{overlap} under $f$ if $f(U_1)\cap f(U_2)\not= \emptyset$, and we say that $U_1$ and $U_2$ \textit{strictly overlap} under $f$ if $f(U_1)=f(U_2)$.

A \textit{global layer ordering} for an isometric folding map $f$ is a function $\lambda_f: A\subset P\times P \to \{-1,1\}$ that records which points of $P$ are above/below which others when folded under $f$, with $\lambda_f(p,q)=1$ meaning that $p$ is below $q$ and $\lambda_f(p,q)=-1$ meaning $p$ is above $q$.  Specifically, $\lambda_f$ is a global layer ordering if the following six properties are satisfied (adopted from \cite{boxpleat}):

\begin{itemize}
\item \textbf{Existence:} The domain $A$ is defined as all $(p,q)\in P\times P$ such that $f(p)=f(q)$. That is, the layer ordering $\lambda_f$ only exists between points that overlap in the folding.

\item \textbf{Antisymmetry:}   $\lambda_f(p,q)=-\lambda_f(q,p)$ for all $(p,q)\in A$.  That is, if $p$ is above $q$ then $q$ is below $p$.

\item \textbf{Transitivity:} If $\lambda_f(p,q)=\lambda_f(q,r)$ then $\lambda_f(p,r)=\lambda_f(p,q)$.  That is, if $q$ is above $p$ and $r$ is above $q$, then $r$ is above $p$.

\item \textbf{Tortilla-Tortilla Property (Consistency):} For any two uncreased,
  simply connected subsets $U_1, U_2\subset P$ that strictly overlap under $f$,
  $\lambda_f$ has the same value for all $(p,q)\in U_1\times U_2$. I.e.,
  if two regions in $P$ completely overlap under $f$, then one must be entirely
  above the other.  This is illustrated in Figure~\ref{fig:tacos}(a).

\item \textbf{Taco-Tortilla Property (Face-Crease Non-crossing):} For any three
  uncreased, simply connected subsets $U_1, U_2, U_3\subset P$ such that (a)
  $U_1$ and $U_3$ are separated by an edge  $e$ in $X_f$ (i.e., adjacent regions
  in $X_f$) and strictly overlap under $f$ and (b) $U_2$ overlaps the edge $e$
  under $f$, then $\lambda_f(p,q)=-\lambda(q,r)$ for any points $(p,q,r)\in
  U_1\times U_2\times U_3$.  I.e., if a region overlaps a nonadjacent
  internal crease, the region cannot lie between the regions adjacent to the
  crease in the folding. This is illustrated in Figure~\ref{fig:tacos}(b).

\item \textbf{Taco-Taco Property (Crease-Crease Non-crossing):} If we have
  uncreased, simply connected adjacent subsets $U_1$ and $V_1$ of $P$ separated
  by a crease $e_1$ in $X_f$ and $U_2$ and $V_2$ separated by a crease $e_2$
  such that the subsets all strictly overlap under $f$ and the creases $e_1$ and
  $e_2$ strictly overlap under $f$, then for any point $(p,q,r,s)\in U_1\times
  V_1\times U_2\times V_2$ either $\{\lambda_f(p,r), \lambda_f(p,s),
  \lambda_f(q,r), \lambda_f(q,s)\}$ are all the same or half are $+1$ and half
  are $-1$.  I.e., if two creases overlap in the folding, either the regions of
  paper adjacent to one crease lie entirely above the regions of paper adjacent
  to the other crease, or the regions of one nest inside the regions of the
  other. This is illustrated in Figure~\ref{fig:tacos}(c).

\end{itemize}

A global layer ordering ensures that if an actual piece of paper $P$ is to be folded according to an isometric folding map $f$ as determined by its crease pattern $X_f$, then this can be done without $P$ intersecting itself. This is the generally-accepted definition of what it means for a crease pattern to be \textit{globally flat-foldable} \cite{boxpleat,GFALOP,origametry,Justin97}. 

An isometric folding map $f$ and global layer ordering $\lambda_f$ determine a dichotomy for the creases of $X_f$, called the \textit{mountain-valley (MV) assignment for $(f,\lambda_f)$}.  Specifically, if a crease $e$ of $X_f$ is bordered by faces $U_1$ and $U_2$ and $p\in U_1$, $q\in U_2$ are close to $e$ with $f(p)=f(q)$, then
\begin{itemize}
\item if the orientation of $U_1$ is preserved under $f$ and $\lambda_f(p,q)=1$, then $e$ is a \textit{valley} crease, and
\item if the orientation of $U_1$ is preserved under $f$ and $\lambda_f(p,q)=-1$, then $e$ is a \textit{mountain} crease.
\end{itemize}
Mountain and valley creases correspond to what we see in physically folded paper, where paper bends in either the $\wedge$ (mountain) or $\vee$ (valley) direction.  A fundamental result about the mountains and valleys that meet at a flat-folded vertex, which we will often use in our proofs, is \textit{Maekawa's Theorem}:

\begin{theorem}[Maekawa's Theorem {\cite[p. 81]{origametry}}]
  If a crease pattern flat-folds around a vertex, the difference between the
  number of mountain folds and the number of valley folds meeting at that vertex
  must be $2$.
\end{theorem}

A generalization of Maekawa and Kawasaki's Theorems that we will need in the proof of Proposition~\ref{prop:onevalley} below is the following:

\begin{theorem}[Justin's Theorem {\cite[p. 116]{Justin97,origametry}}]
Let $\gamma$ be a simple, closed, vertex-avoiding curve on a flat-foldable crease pattern.  Let $\alpha_i$ be the signed angles, in order, between the consecutive creases that $\gamma$ crosses, for $1\leq i\leq 2n$. Also let $M$ and $V$ be the number of mountain and valley creases, respectively, that $\gamma$ crosses.  Then,
\begin{equation}\label{eq:Justin}
\alpha_1+\alpha_3+\cdots+\alpha_{2n-1}\equiv \alpha_2+\alpha_4+\cdots+\alpha_{2n} \equiv \frac{M-V}{2}\pi \mod 2\pi.
\end{equation}
\end{theorem}

Computing a global layer ordering, or determining than none exist, for a given isometric folding map is computationally intensive and the main reason why the global flat-foldability problem is NP-hard \cite{Bern}. A useful tool that we will employ (specifically in the proof of Lemma~\ref{lem:problem} below) is the \textit{superposition net} (or \textit{s-net} for short) that Justin introduced in \cite{Justin97}. The s-net is a superset of the crease pattern $X_f$ of an isometric folding map given by $S_f=f^{-1}(f(X_f))$.  That is, $S_f$ is the pre-image of the folded image of the crease pattern. This is helpful because the points of $S_f$ are places where the tortilla-tortilla, taco-tortilla, or taco-taco properties might fail. See \cite[Section 6.5]{origametry} for more details.

\subsection{Rule 110 and our conventions}\label{sec:Rule110}

Rule 110 is an elementary (1-dimensional) cellular automaton using the rule table shown in Figure~\ref{fig:rule110}.
We model a $1$ as TRUE and a $0$ as FALSE, or black and white pixels, as in Figure~\ref{fig:rule110} where each 1-dimensional state of the cellular automaton is stacked vertically to show the step-by-step evolution of the system.   Note that if all inputs are set to
$0$ the automaton stays constant at $0$.

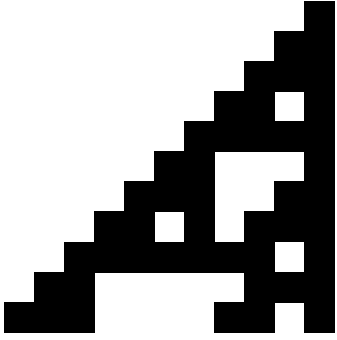
\newsavebox{\tempbox}

\begin{figure}
\centering
\sbox{\tempbox}{
\begin{tikzpicture}[scale=.4]
\fill (0,0) rectangle (1,1);
\fill (-1,-1) rectangle (1,0);
\fill (-2,-2) rectangle (1,-1);
\fill (-3,-3) rectangle (-1,-2); \fill (0,-3) rectangle (1,-2);
\fill (-4,-4) rectangle (1,-3);
\fill (-5,-5) rectangle (-3,-4); \fill (0,-5) rectangle (1,-4);
\fill (-6,-6) rectangle (-3,-5); \fill (-1,-6) rectangle (1,-5);
\fill (-7,-7) rectangle (-5,-6); \fill (-4,-7) rectangle (-3,-6); \fill (-2,-7) rectangle (1,-6);
\fill (-8,-8) rectangle (-1,-7); \fill (0,-8) rectangle (1,-7);
\fill (-9,-9) rectangle (-7,-8); \fill (-2,-9) rectangle (1,-8);
\fill (-10,-10) rectangle (-7,-9); \fill (-3,-10) rectangle (-1,-9); \fill (0,-10) rectangle (1,-9);
\end{tikzpicture}
}% 
\subfloat{\usebox{\tempbox}}% 
%\qquad
\subfloat{% % %
    \vbox to \ht\tempbox{%
      \vfil
\hspace{-1.5in}\begin{tabular}{r|c|c|c|c|c|c|c|c}
    Current pattern & 111&110&101&100&011&010&001&000 \\
    \hline
    New state & 0 & 1 & 1 & 0 & 1 & 1 & 1 & 0 
  \end{tabular}
      \vfil}}%

%\begin{tblr}{Q[m]Q[t]}
%
%\begin{tabular}{r|c|c|c|c|c|c|c|c}
%    Current pattern & 111&110&101&100&011&010&001&000 \\
%    \hline
%    New state & 0 & 1 & 1 & 0 & 1 & 1 & 1 & 0 
%  \end{tabular}
%  &
%  \begin{tikzpicture}[scale=.5]
%\fill (0,0) rectangle (1,1);
%\fill (-1,-1) rectangle (1,0);
%\fill (-2,-2) rectangle (1,-1);
%\fill (-3,-3) rectangle (-1,-2); \fill (0,-3) rectangle (1,-2);
%\fill (-4,-4) rectangle (1,-3);
%\fill (-5,-5) rectangle (-3,-4); \fill (0,-5) rectangle (1,-4);
%\fill (-6,-6) rectangle (-3,-5); \fill (-1,-6) rectangle (1,-5);
%\fill (-7,-7) rectangle (-5,-6); \fill (-4,-7) rectangle (-3,-6); \fill (-2,-7) rectangle (1,-6);
%\fill (-8,-8) rectangle (-1,-7); \fill (0,-8) rectangle (1,-7);
%\fill (-9,-9) rectangle (-7,-8); \fill (-2,-9) rectangle (1,-8);
%\fill (-10,-10) rectangle (-7,-9); \fill (-3,-10) rectangle (-1,-9); \fill (0,-10) rectangle (1,-9);
%\end{tikzpicture}
%\end{tblr}
\caption{The table for Rule 110 and ten rows of its evolution from a single TRUE pixel.}\label{fig:rule110}
\end{figure}

We will simulate Rule 110 in an origami crease pattern by establishing conventions by which the creases can be interpreted as storing and manipulating Boolean variables.
In a similar strategy to prior work on origami complexity
\cite{boxpleat,rigidcom,Bern}, we use \textit{directed pleats} (sequences of
parallel mountain/valley crease lines) to send TRUE/FALSE signals across the
folded paper; we call such directed pleats \textit{wires}.  Our wires are
triplets of parallel creases, with a mandatory mountain in the middle and
optional valleys to the left and right. We will orient our crease patterns so
that the direction of all wires is in the ``downward," decreasing $y$ direction
in $\mathbb{R}^2$.  The value of a wire is decided as follows:  If the pleat is
folded to the right relative to its direction then it is FALSE; if it is folded
to the left then it is TRUE. The information in a wire consists of the choice of
which valley fold is used.

\begin{figure}
  \begin{center}
\begin{tikzpicture}[yscale=0.5]  % key
\draw[mountain] (0,0) -- (0.5,0) node[right] {mountain};
\draw[valley] (0,-1) -- (0.5,-1) node[right] {valley};
\draw[mountain_opt] (0,-2) -- (0.5,-2) node[right] {optional mountain};
\draw[valley_opt] (0,-3) -- (0.5,-3) node[right] {optional valley};
\draw[mountain] (5,0) -- (7.5,-2);
\draw[valley_opt] (6,0) -- (8.5,-2);
\draw[valley_opt] (4,0) -- (6.5,-2);
\draw [-stealth](7,-.2) -- (8.5,-1.4);
\node at (6.5,-2.5){a wire};
\end{tikzpicture}
\end{center}
\caption{A guide to the mountain/valley labeling conventions uses in our figures and an example of a wire.}
\label{fig:mvlabels}
\end{figure}

The labeling conventions we will use in the crease pattern figures in this paper are shown in Figure~\ref{fig:mvlabels}, with solid lines depicting mountain creases and dashed lines being valleys, which is standard in much origami literature. Also, black creases will be mandatory and blue will be optional. An illustration of a wire and its direction is also shown in Figure~\ref{fig:mvlabels}.

In this paper we do not prove that the given crease patterns will fold flat in
accordance to the stated operation; that can be verified directly via folding.
In our proofs we will simply verify that the given crease patterns will
\emph{not} fold in other ways.

\subsection{Fundamental results}

\begin{definition}
  In a crease pattern with optional creases, a crease is \emph{active} in a
  flat-folding if it is used.
\end{definition}

\begin{definition}
  In this paper we will be working on an infinite triangular grid of triangles
  with side-length $1$.  For any line on the grid, the \emph{next line} in some
  direction is the closest line in that direction which is parallel to it; if
  $\ell,\ell'$ are two lines and $\ell'$ is the next line then $\ell$ and
  $\ell'$ are \emph{consecutive}.

  A \emph{wire} is three consecutive creases, where the middle crease is a
  mountain crease, and the two outer creases are optional valley creases.
  
  A \emph{gadget} is a subset of a crease pattern in a region of the plane bounded by a 
  simple closed curve such that the only creases that intersect the boundary are
  wires. %\Tom{Do we need to also require that the \textit{only} wires in a    gadget are those that intersect the boundary?} \inna{No, I don't think so.    It's possible to have a gadget that's more than one cell. Consider, for    instance, the Sierpinski thing that we defined.  Why couldn't that be a    gadget?}
\end{definition}

Note that in order for a wire to have a well-defined Boolean value, we need exactly one of its optional valley creases to be active. However, by themselves any wire could have all three of its creases folded, and thus we want the gadgets that wires enter and exit to force the wires to have one, and only one active crease.  The below Proposition will help ensure this.

%\Tom{Do we really need Proposition 6?  It seems to only be used as a way to guarantee that when our logic gadgets fold flat, each wire must have exactly one of the optional creases activated.  But all of our gadgets (except the hexagon and triangle twists, which are handled differently) have prescribed inputs and exactly one output.  Thus, if all but one of the wires of a gadget are determined to be one active M and one active V, then the sole remaining wire cannot have only 1 or 3 active creases, else the boundary of the gadget would cross an odd number of creases, which is impossible for flat origami (by the 2-colorability of faces).  So the output wire must have the mandatory mountain crease and exactly one of the valleys.  Maybe this little argument should be our Proposition 6?  Or am I missing something? (Oh wait --- I see that our Intersector gadgets have two output wires.  Hmmm.)}

\begin{proposition} \label{prop:onevalley}
  Let $G$ be a gadget with angles $\theta_1,\ldots,\theta_k$ between consecutive boundary
  wires (calculated as one transverses the boundary counter-clockwise).  Suppose
  that for all nonempty proper subsets $J\subsetneq \{1,\ldots,k\}$ we have
  \[\sum_{j\in J} \theta_j \not\neq 0 \pmod{\pi}.\]
  If we pick optional creases to make a sub-crease pattern $G'\subset G$ be flat-foldable then each wire in $G'$ contains exactly one active valley fold.  
\end{proposition}

%\Tom{I actually do not believe this Proposition.  I think it is too general.}

%\Tom{But all of our gadgets can be shown to require all the wires having exactly one valley crease activated. But the argument is only general when there is one output wire, where a 2-coloring along the boundary argument suffices. For the Intersectors, the reason why their wires work is because the angles between the wires are $\pi/3, 2\pi/3, \pi/3, 2\pi/3$, and having the two output wires with some combination of zero or two active valleys would break Justin's Theorem (but a small case-by-case argument needs to be made, I think, to cover the possible values of $M-V$ along the boundary). Anyway, I think Proposition 6 needs to be reworked to, maybe, handle just the single-output-wire case.  Then maybe the Intersector gadget case should just be handled in Section 3.4.}

\begin{proof}
  The angle between consecutive creases in a wire is $0$.  In addition, each
  wire contains at least one active crease (the mountain crease).  Applying
  Justin's Theorem to our sub-crease pattern $G'$ with $\gamma$ the gadget's boundary and $\alpha_1,\ldots,\alpha_n$ the angles in counterclockwise order between the creases in $G'$ that $\gamma$ crosses,%(\cite{Justin97},\cite[Chapter 6.4]{origametry}) 
 we have $\alpha_1+\alpha_3+\cdots + \alpha_{n-1}\equiv \alpha_2 + \alpha_4 + \cdots + \alpha_n \equiv  m\pi \pmod{2\pi}$ for some constant $m$. Now, each $\alpha_i$ equals some $\theta_j$ or is zero, so our assumption  that no proper subset of the $\theta_j$'s adds up to a multiple of $\pi$ must also be true for the non-zero $\alpha_i$'s. Therefore, the non-zero $\alpha_i$'s all have either $i$ being odd or all have $i$ being even. But if both (or neither) of the valley folds in a wire of $G$ were active in $G'$, then the angles $\theta_j$ and $\theta_{j+1}$ surrounding the wire would appear among the $\alpha_i$ angles with one having an odd $i$ and one an even $i$, which is a contradiction. Thus every wire in $G'$ contains exactly one active valley fold.
\end{proof}

As a corollary we get:
\begin{corollary} \label{cor:onevalley_int}
  Let $G$ be a gadget with two input wires and two output wires, with
  consecutive angles between wires adding up to $2\pi$.  If we know that the
  input wires each contain exactly one active valley fold, the output wires must
  each contain exactly one active valley fold.
\end{corollary}

%In our analyses of gadgets we will repeatedly use the following two theorems:
%\begin{theorem}[Maekawa's Theorem {\cite[Maekawa's Theorem, p. 81]{origametry}}]
%  If a crease pattern flat-folds around a vertex, the difference between the
%  number of mountain folds and the number of valley folds meeting at that vertex
%  must be $2$.
%\end{theorem}
%
%\begin{theorem}[Kawasaki's Theorem {\cite[Theorem 5.37]{origametry}}]
%  If a crease pattern flat-folds around a vertex, the sum of alternating angles
%  between the creases meeting at the vertex must be $\pi$.
%\end{theorem}

\section{Logic gates and other gadgets}\label{sec:gates}

In this section we show how to construct logical gates (AND, OR, NAND, NOR, NOT) as well as intersector, twist, and eater gadgets via origami with optional creases. These will form the building blocks of our Rule 110 flat origami simulation. Our over-all scheme is to build our crease pattern on a triangle lattice, and so most of our gadgets will possess triangle or hexagonal symmetry. We also include gadgets for some logic gates (AND and NOT) merely for completeness, as they are not used in the final construction.

For the two-input logic gates we assume that we
are given two wires at an angle of $2\pi/3$ with information coming ``in'' to
the gate from the positive $y$-direction.  The output is a crease at an angle of $2\pi/3$ with each of the
inputs.  The NOT gate (Section~\ref{sec:not}) is somewhat strange, as it
requires an ``auxilliary'' pleat which is not affected by the gate.  In
addition, in Section~\ref{sec:intersect}, we show that it is possible to
``intersect'' two wires which meet at an angle of $\pi/3$ without affecting
their values.  Sections~\ref{sec:twist} and \ref{sec:eat} contain the twist and eater gadgets. %some basic elements which we will use later in the construction of Rule 110.

In the below Propositions, we say that a logic gadget \textit{works} if the values of the input wires force the correct output wire value in accordance to the desired logic gate.

\subsection{NOR and NAND}

\begin{figure}[h]
  \centering
  \begin{minipage}{.5\textwidth}
    \centering
\begin{tikzpicture}
% first draw the bidirectional twist
\draw[mountain] (30:1) -- (30:3.5)
                (150:1) -- (150:3.5)
                (0,-1) -- (0,-3.5)
                (30:1) -- (150:1) -- (0,-1) -- cycle;
\draw[valley_opt] (30:1) -- +(0,-4)
                  (30:1) -- +(150:4)
                  (150:1) -- +(30:4)
                  (150:1) -- +(0,-4)
                  (0,-1) -- +(30:4)
                  (0,-1) -- +(150:4);
% then add control folds 
                  
\draw[valley_opt] (0,0) -- (30:1)
                  (0,0) -- (150:1)
                  (0,0) -- (0,-1)
                  (60:{sqrt(3)/3}) -- (0,1)
                  (0,0) -- (120:{sqrt(3)/3})
                  (0,0) -- (-120:{sqrt(3)/3})
                  (-{sqrt(3)/3},0) -- (-150:1);
\draw[mountain_opt] (120:{sqrt(3)/3}) -- (0,1)
                    (0,0) -- (60:{sqrt(3)/3})
                    (0,0) -- ({-sqrt(3)/3},0)
                    (-150:1) -- (-120:{sqrt(3)/3})
                    ;

%input/output labels
\path (0,0.5) +(150:3) node[rotate=-30] (Fin1) {FALSE};
\draw[->] (Fin1) -- +(-30:1);
\path (0,-0.5) +(150:3.5) node[rotate=-30] (Tin1) {TRUE};
\draw[->] (Tin1) -- +(-30:1);
\path (0,-0.5) +(30:3.5) node[rotate=30] (Fin2) {FALSE};
\draw[->] (Fin2) -- +(-150:1);
\path (0,0.5) +(30:3) node[rotate=30] (Fin2) {TRUE};
\draw[->] (Fin2) -- +(-150:1);
\path (-30:0.5) -- +(0,-3) node[rotate=-90] (Fout) {FALSE};
\path (-150:0.5) -- +(0,-3) node[rotate=-90] (Tout) {TRUE};
\draw[->] (Fout) -- +(0,-1);
\draw[->] (Tout) -- +(0,-1);

% important points
\node [circle, fill=black, draw, label=below right:{$\scriptscriptstyle X$}, inner sep=0.13em] at
(-30:1) {};
\node [circle, fill=black, draw, label={[label distance=-4pt]below
  right:{$\scriptscriptstyle  O$}}, inner sep=0.13em] at (0,0) {};
\node [circle, fill=black, draw, label={above:{$\scriptscriptstyle  Y$}}, inner sep=0.13em] at (0,1)
{};
\node [circle, fill=black, draw, label={below left:{$\scriptscriptstyle  Z$}},
inner sep=0.13em] at (-150:1) {};
\node [circle, fill=black, draw, label={above:{$\scriptscriptstyle  Y'$}}, inner sep=0.13em] at (150:1) {};
\node [circle, fill=black, draw, label={above:{$\scriptscriptstyle  X'$}}, inner
sep=0.13em] at (30:1) {};
%\node [circle, fill=black, draw, label={-90:{$\scriptscriptstyle  A'$}}, inner sep=0.13em] at (120:{sqrt(3)/3}) {};
\node [circle, fill=black, draw, inner sep=0.13em] at (120:{sqrt(3)/3}) {};
%\node at (-.08,.65,0) {$\scriptscriptstyle  A'$};
%[circle, fill=black, draw, label={-90:{$\scriptscriptstyle  A'$}}, inner sep=0.13em] at (120:{sqrt(3)/3}) {};
%\node [circle, fill=black, draw, label={180:{$\scriptscriptstyle  A$}}, inner sep=0.13em] at ({-sqrt(3)/3},0) {};
\node [circle, fill=black, draw, inner sep=0.13em] at ({-sqrt(3)/3},0) {};
%\node at ({-sqrt(3)/3-0.2},0) {$\scriptscriptstyle  A$};
%[circle, fill=black, draw, label={180:{$\scriptscriptstyle  A$}}, inner sep=0.13em] at ({-sqrt(3)/3},0) {};

\end{tikzpicture}
\captionof{figure}{NOR gate}
\label{fig:altnor}
\end{minipage}%
\begin{minipage}{.5\textwidth}
  \centering
\begin{tikzpicture}[scale=0.85]
% first draw the bidirectional twist
\draw[mountain] (30:1) -- (30:3.5)
                (150:1) -- (150:3.5)
                (0,-1) -- (0,-3.5)
                (30:1) -- (150:1) -- (0,-1) -- cycle;
\draw[valley_opt] (30:1) -- +(0,-4)
                  (30:1) -- +(150:4)
                  (150:1) -- +(30:4)
                  (150:1) -- +(0,-4)
                  (0,-1) -- +(30:4)
                  (0,-1) -- +(150:4);

% then add control folds
\draw[valley_opt] (0,0) -- (30:1)
                  (0,0) -- (150:1)
                  (0,0) -- (0,-1)
                  (120:{sqrt(3)/3}) -- (0,1)
                  (0,0) -- (60:{sqrt(3)/3})
                  (0,0) -- (-60:{sqrt(3)/3})
                  ({sqrt(3)/3},0) -- (-30:1);
\draw[mountain_opt] (60:{sqrt(3)/3}) -- (0,1)
                    (0,0) -- (120:{sqrt(3)/3})
                    (0,0) -- ({sqrt(3)/3},0)
                    (-30:1) -- (-60:{sqrt(3)/3})
                    ;

% input/output labels
\path (0,0.5) +(150:3) node[rotate=-30] (Fin1) {FALSE};
\draw[->] (Fin1) -- +(-30:1);
\path (0,-0.5) +(150:3.5) node[rotate=-30] (Tin1) {TRUE};
\draw[->] (Tin1) -- +(-30:1);
\path (0,-0.5) +(30:3.5) node[rotate=30] (Fin2) {FALSE};
\draw[->] (Fin2) -- +(-150:1);
\path (0,0.5) +(30:3) node[rotate=30] (Fin2) {TRUE};
\draw[->] (Fin2) -- +(-150:1);
\path (-30:0.5) -- +(0,-3) node[rotate=-90] (Tout) {FALSE};
\path (-150:0.5) -- +(0,-3) node[rotate=-90] (Fout) {TRUE};
\draw[->] (Fout) -- +(0,-1);
\draw[->] (Tout) -- +(0,-1);

\end{tikzpicture}
\captionof{figure}{NAND gate}
\label{fig:altnand}
\end{minipage}
\end{figure}

\begin{proposition} \label{prop:nor}
  The NOR gate (Figure~\ref{fig:altnor}) works.
\end{proposition}

\begin{proof}
  First, suppose that the upper-right input is TRUE.  Thus the crease at $\pi/6$
  from $X$ is active, and therefore so is the crease at $-5\pi/6$ from $X$, since only one of the optional valley creases through $X$ may be used by Kawasaki's Theorem.
  That is, the crease at $-\pi/2$ cannot be active, and thus the output is FALSE.
  
    \begin{figure} 
    \begin{tikzpicture}[scale=2]
      \node [circle, fill=black, draw, inner sep=0.15em] at (0,0) {};
      \draw[valley] (0,0) -- (0,-.8)
      (0,0) -- (150:1)
      (0,0) -- (120:1);
      \draw[mountain] (0,0) -- (60:1);

      \node at (90:0.5) {$\frac{\pi}{3}$};
      \node at (135:0.7) {$\frac{\pi}{6}$};
      \node at (-150:0.5) {$\frac{2\pi}{3}$};
      \node[align=center] at (-15:0.5) {$\frac{5\pi}{6}$};
    \end{tikzpicture}
    \caption{An impossible MV assignment in a NOR gate}\label{eq:impossible}
  \end{figure}
  
  Before considering the other two cases, a basic observation.  Consider the
  following crease pattern around a point shown in Figure~\ref{eq:impossible}.
  This will not fold flat, since the wedge which is $\pi/6$ wide is too narrow:
  the wedges on either side of it will collide if we try to flat-fold it.  Thus
  such a configuration is impossible. 
  
  Now suppose that both inputs are FALSE.  By Kawasaki's Theorem around $Y$, the
  fold $YY'$ must be active while none of the others can be (since the fold at
  $\pi/6$ is active and the fold at $5\pi/6$ is not).  By Maekawa's Theorem
  neither of the other two valley folds through $Y'$ can be active.  This means
  that the folds at $\pi/3$, $2\pi/3$ and $5\pi/6$ through $O$ are not active.
  From this it follows that none of the folds through $O$ can be active, since
  otherwise they must all be and then they exactly form the impossible
  configuration in Figure~\ref{eq:impossible}.  Thus the folds at $0$,
  $\pi/3$ and $\pi/2$ through $Z$ are not active, and therefore the fold at
  $-\pi/2$ through $Z$ cannot be active.  Thus the output is TRUE, as desired.

  Lastly, suppose that the left-hand input is TRUE and the right-hand input is
  FALSE.  Then the folds at $\pi/6$ and $5\pi/6$ through $Y$ are active, and
  therefore, for Kawasaki's and Maekawa's Theorems to hold around $Y$, the folds at $-\pi/3$
  and $-2\pi/3$ must be active and $YY'$ and $YX'$ not active. Thus the folds at $\pi/3$ and $2\pi/3$ through
  $O$ are active.  In order for Kawasaki's Theorem to hold around $O$, we must
  have the fold at $-\pi/2$ active, and exactly one of the two folds $OY'$ and
  $OX'$ active.  As the configuration in Figure~\ref{eq:impossible} is
  impossible, the fold $OY'$ cannot be the one that is active, and thus $OX'$
  must be the one that is active.  By Maekawa's Theorem the fold $Y'Z$ must
  therefore be active, and thus also the fold at $-\pi/2$ through $Z$, showing
  that the output is FALSE, as desired.

\end{proof}

\begin{proposition} \label{prop:nand}
  The NAND gate (Figure~\ref{fig:altnand}) works.
\end{proposition}

\begin{proof}
  The NAND gate is a reflection in a vertical line of the NOR gate; since
  reflection is orientation-reversing in the horizontal direction and
  orientation-preserving in the vertical direction, by Proposition~\ref{prop:nor},
  the output of the NAND gate with inputs $A$ and $B$ is
  \[\neg ({\neg A} \scnor {\neg B}) = {\neg A} \scor {\neg B} = A \scnand B,\]
  as desired.
\end{proof}

\subsection{OR and AND}

\begin{figure}
\centering
\begin{minipage}{.5\textwidth}
  \centering

\begin{tikzpicture}
% first draw the bidirectional twist
\draw[mountain] (30:0) -- (30:3.5)
                (150:0) -- (150:3.5)
                (0,0) -- (0,-3.5);
\draw[valley_opt] (30:1) -- +(0,-4)
                  (30:1) -- +(150:4)
                  (150:1) -- +(30:4)
                  (150:1) -- +(0,-4)
                  (0,-1) -- +(30:4)
                  (0,-1) -- +(150:4);
% then add control folds
\draw [valley_opt] (0,0) -- (0,1)
                   (0,0) -- (-30:1)
                   (30:0.5) -- (-30:1)
                   (30:2) -- (-30:1)
                   (0,0) -- (-150:1);
\draw[mountain_opt] (150:1) -- (-150:1)                   
                    (30:1) -- (-30:1)
                    (30:0.5) -- (0,1)
                    (30:2) -- (0,1)
                    (-30:1) -- (0,-1);

%input/output labels
\path (0,0.5) +(150:3) node[rotate=-30] (Fin1) {FALSE};
\draw[->] (Fin1) -- +(-30:1);
\path (0,-0.5) +(150:3.5) node[rotate=-30] (Tin1) {TRUE};
\draw[->] (Tin1) -- +(-30:1);
\path (0,-0.5) +(30:3.5) node[rotate=30] (Fin2) {FALSE};
\draw[->] (Fin2) -- +(-150:1);
\path (0,0.5) +(30:3) node[rotate=30] (Fin2) {TRUE};
\draw[->] (Fin2) -- +(-150:1);
\path (-30:0.5) -- +(0,-3) node[rotate=-90] (Tout) {FALSE};
\path (-150:0.5) -- +(0,-3) node[rotate=-90] (Fout) {TRUE};
\draw[->] (Fout) -- +(0,-1);
\draw[->] (Tout) -- +(0,-1);

\node [circle, fill=black, draw, label={below left:{$\scriptscriptstyle  Z$}},
inner sep=0.13em] at (-150:1) {};
\node [circle, fill=black, draw, label={below right:{$\scriptscriptstyle  X$}}, inner sep=0.13em] at (-30:1) {};
\node [circle, fill=black, draw, label={above:{$\scriptscriptstyle  Y$}},
inner sep=0.13em] at (0,1) {};
\node [circle, fill=black, draw, label={above:{$\scriptscriptstyle  Y'$}},
inner sep=0.13em] at (150:1) {};
\node [circle, fill=black, draw, label={above:{$\scriptscriptstyle  X'$}}, inner
sep=0.13em] at (30:1) {};
\node [circle, fill=black, draw, label={below right:{$\scriptscriptstyle  Z'$}}, inner sep=0.13em] at (0,-1) {};
%\node [circle, fill=black, draw, label={below left:{$\scriptscriptstyle  O$}}, inner sep=0.13em] at (0,0) {};
\node [circle, fill=black, draw, inner sep=0.13em] at (0,0) {};
\node at (-.18,-.25) {$\scriptscriptstyle  O$};
%[circle, fill=black, draw, label={below left:{$\scriptscriptstyle  O$}}, inner sep=0.13em] at (0,0) {};

\end{tikzpicture}
\captionof{figure}{OR gate}
  \label{fig:new-or}
\end{minipage}%
\begin{minipage}{.5\textwidth}
  \centering

\begin{tikzpicture}[scale=0.85]
% first draw the bidirectional twist
\draw[mountain] (30:0) -- (30:3.5)
                (150:0) -- (150:3.5)
                (0,0) -- (0,-3.5);
\draw[valley_opt] (30:1) -- +(0,-4)
                  (30:1) -- +(150:4)
                  (150:1) -- +(30:4)
                  (150:1) -- +(0,-4)
                  (0,-1) -- +(30:4)
                  (0,-1) -- +(150:4);
% then add control folds
\draw [valley_opt] (0,0) -- (0,1)
                   (0,0) -- (-30:1)
                   (150:0.5) -- (-150:1)
                   (150:2) -- (-150:1)
                   (0,0) -- (-150:1);
\draw[mountain_opt] (30:1) -- (-30:1)                   
                    (-150:1) -- (0,-1)
                    (-150:1) -- (150:1)
                    (150:0.5) -- (0,1)
                    (150:2) -- (0,1);

%input/output labels
\path (0,0.5) +(150:3) node[rotate=-30] (Fin1) {FALSE};
\draw[->] (Fin1) -- +(-30:1);
\path (0,-0.5) +(150:3.5) node[rotate=-30] (Tin1) {TRUE};
\draw[->] (Tin1) -- +(-30:1);
\path (0,-0.5) +(30:3.5) node[rotate=30] (Fin2) {FALSE};
\draw[->] (Fin2) -- +(-150:1);
\path (0,0.5) +(30:3) node[rotate=30] (Fin2) {TRUE};
\draw[->] (Fin2) -- +(-150:1);
\path (-30:0.5) -- +(0,-3) node[rotate=-90] (Tout) {FALSE};
\path (-150:0.5) -- +(0,-3) node[rotate=-90] (Fout) {TRUE};
\draw[->] (Fout) -- +(0,-1);
\draw[->] (Tout) -- +(0,-1);

\end{tikzpicture}
\centering
\captionof{figure}{AND gate}
  \label{fig:new-and}

\end{minipage}
\end{figure}

\begin{proposition} \label{prop:new-or}
  The OR gate (Figure~\ref{fig:new-or}) works.
\end{proposition}

\begin{proof}
  As before, by Proposition~\ref{prop:onevalley} exactly one of the output
  valley folds must be active.

  First, suppose that the left-hand input is FALSE.  Then the crease at $5\pi/6$
  from $Y$ is not active, so the only creases out of $Y$ that can be active are
  $YY'$ and the crease at $\pi/6$, which are either both active or not.  Thus if
  the right-hand input is TRUE then none of the creases through $Y$ are
  active.  Since none of the creases through $Y$ are active, $YY'$ is not
  active, and thus neither is $Y'Z$.  Thus the crease at $-\pi/2$ through $Z$
  cannot be active, and the output is TRUE.  On the other hand, if the
  right-hand input is FALSE then $YY'$ must be active, and thus therefore so
  must $Y'Z$.  Thus the crease at $-\pi/2$ through $Z$ must be active (since the
  crease at $5\pi/6$ through $Z$ is also active), and the input must be FALSE.

  Now suppose that the left-hand input is TRUE.  Then the crease at $5\pi/6$
  through $Z$ is not active.  Thus $ZO$ and $ZZ'$ cannot be active, and the
  crease at $-\pi/2$ through $Z$ is active if and only if $ZY'$ is active, which is
  active if and only if $YY'$ is active.  Thus to show that the output is TRUE
  it suffices to check that $YY'$ is not active. If the right-hand input is
  TRUE then the crease at $\pi/6$ through $Y$ is not active; by Kawasaki's theorem 
  $YY'$ must not be active in this case, as desired.  If the right-hand input is
  FALSE then the crease at $\pi/6$ through $Y$ is active.  If $YY'$ were active
  then by Kawasaki's theorem $YO$ must be active, and none of the other creases
  through $Y$ can be.  But this violates Maekawa's theorem, since it has $4$
  valley and no mountain creases meeting at a point.  Thus $YY'$ must not be
  active in this case either, as desired.
\end{proof}

\begin{proposition} \label{prop:new-and}
  The AND gate (Figure~\ref{fig:new-and}) works.
\end{proposition}

\begin{proof}
  The AND gate is a reflection of the OR gate.  By the same logic as in the
  proof of Proposition~\ref{prop:nand}, since the OR gate works, so does the AND.
\end{proof}

\subsection{NOT gate} \label{sec:not}

This NOT gate is not used in our origami construction of Rule 110, but we include it for completeness.

\begin{figure}[h]
  \centering
  \begin{tikzpicture}
  \draw [mountain] (0,1) -- (0,3) (0,-1) -- (0,-3);
  \draw [mountain] (-3,.5) -- (-.5,.5) (-.5,-.5) -- (-3,-.5) (3,.5) -- (.5,.5) (.5,-.5) -- (3,.-.5);
  \draw [mountain] (0,1) -- (-.5,.5) -- (-.5,-.5) -- (0,-1) -- (.5,-.5) -- (.5,.5) -- cycle;
  \draw [valley] (-3,1) -- (-1,1) (1,1) -- (3,1) (-3,-1) -- (-1,-1) (1,-1) -- (3,-1);
  \draw [valley_opt] (1,3) -- (1,1) (1,-1) -- (1,-3) (-1,3) -- (-1,1) (-1,-1) -- (-1,-3) (-1,1) -- (1,1) (-1,-1) -- (1,-1);
  \draw [valley_opt] (.5,.5) -- (1.5,-.5) (.5,-.5) -- (1.5,.5) (-.5,.5) -- (-1.5,-.5) (-.5,-.5) -- (-1.5,.5);
  \draw [valley_opt] (.5,.5) -- (1,1) (.5,-.5) -- (1,-1) (-.5,.5) -- (-1,1) (-.5,-.5) -- (-1,-1);
  \draw [mountain_opt] (1,1) -- (1.5,.5) (1,-1) -- (1.5,-.5) (-1,1) -- (-1.5,.5) (-1,-1) -- (-1.5,-.5);
  %\draw [opt] (-1,1) -- (1,1);

  %input/output labels
  \node[rotate=-90] at  (0.5,3)  (Fin) {FALSE};
  \draw[->] (Fin) -- +(0,-1);
  \node[rotate=-90] at (-0.5,3) (Tin) {TRUE};
  \draw[->] (Tin) -- +(0,-1);
  \path (-30:0.5) -- +(0,-3) node[rotate=-90] (Tout) {FALSE};
  \path (-150:0.5) -- +(0,-3) node[rotate=-90] (Fout) {TRUE};
  \draw[->] (Fout) -- +(0,-1);
  \draw[->] (Tout) -- +(0,-1);

  % important points
  %\node [circle, fill=black, draw, label=above right:{$\scriptscriptstyle X$}, inner sep=0.13em] at
  %(1,0) {};
    \node [circle, fill=black, draw, label=below right:{$\scriptscriptstyle Y$}, inner sep=0.13em] at
  (0,-1) {};
  %\node [circle, fill=black, draw, label=above left:{$\scriptscriptstyle Z$}, inner sep=0.13em] at
  %(-1,0) {};
    \node [circle, fill=black, draw, label=above right:{$\scriptscriptstyle X$}, inner sep=0.13em] at
  (0,1) {};
  \node [circle, fill=black, draw, label=above right:{$\scriptscriptstyle A$}, inner sep=0.13em] at
  (1,1) {};
  \node [circle, fill=black, draw, label=below right:{$\scriptscriptstyle B$}, inner sep=0.13em] at
  (1,-1) {};
   \node [circle, fill=black, draw, label=below left:{$\scriptscriptstyle C$}, inner sep=0.13em] at
  (-1,-1) {};
   \node [circle, fill=black, draw, label=above left:{$\scriptscriptstyle D$}, inner sep=0.13em] at
  (-1,1) {};
 \node [label=above right:{$\scriptstyle L_1$}] at (1,2) {};
 \node [label=above right:{$\scriptstyle L_2$}] at (2,1) {};  
 \node [label=below right:{$\scriptstyle L_3$}] at (2,-1) {};
 \node [label=above left:{$\scriptstyle L_4$}] at (-1,2) {};

  \end{tikzpicture}
  \caption{NOT gate}
  \label{fig:not2}
\end{figure}

\begin{proposition}
  The NOT gate (Figure~\ref{fig:not2}) works.
\end{proposition}

\begin{proof} 
% {\bf NOTE: The below proof does not work with the new version in Figure~\ref{fig:not2}. But it does serve as a proof that twists act logically like we want them to.}
% Note that the vertices $W, X, Y$ and $Z$ can only have four of their five creases used to fold flat, which thus include the three required mountain creases and one of the optional valley creases at each vertex.

% Suppose that the input is TRUE. Then the part of the optional valley $L_1$ above $X$ is used, meaning that the half of $L_1$ below $X$ is not used and the half of $L_2$ to the right of $W$ is not used (since it is blocked by $L_1$ above $X$).  Therefore in order for $W$ to fold flat, it must use the half of $L_2$ to the left of $W$.  Similarly, the half of $L_3$ to the right of $Y$ is blocked, so the right half of $L_3$ is used, and the half of $L_4$ below $Z$ is blocked, so the above half is used.  Thus the crease pattern is a classic square twist \cite[Figure 6.12]{origametry} that twists the inner diamond in the clockwise direction, and the output is FALSE.

% By symmetry, if the input is FALSE then the crease pattern will, again, be a classic square twist that twists the inner diamond in the counterclockwise direction, making the output TRUE.
% \vspace{.1in}

% Actual proof of the new NOT gate:

Suppose the input is TRUE.  Then the crease $L_4$ above the point $D$ is not used, which implies that all of $L_2$ to the left of point $X$ is used and $XA$ is not used (in order to make $X$ flat-foldable).  Thus, since the creases $L_1$ above and $L_2$ to the right of $A$ are both used, we have that both of the short diagonals below $A$ are used.  This implies that only the lower-left-to-upper-right longer diagonal between $A$ and $B$ is used (this diagonal is forced by the short diagonals below $A$; the other diagonal between $A$ and $B$ cannot also be used because we cannot have a degree-4 vertex made of only valley creases).  This implies that the two short diagonals above point $B$ are not used, which means the crease $L_3$ to the right of point $Y$ is used and $L_1$ below $B$ is not used.  Therefore the output is FALSE.  By the left-right mirror symmetry of this crease pattern, if the input is FALSE the output must be TRUE.
\end{proof}

\subsection{Intersector} \label{sec:intersect}

Intersector gadgets will be placed wherever two wires need to cross each other on our sheet of paper. They will ensure that the Boolean signals of the wires will be preserved through the intersection.

\begin{figure}[h]
  \centering
  \begin{minipage}{.5\textwidth}
    \centering
\begin{tikzpicture}

  \draw[mountain] (120:4) -- (-60:3);
  \draw[valley_opt] (120:4)++(1,0) -- +(-60:6);
  \draw[valley_opt] (120:4)++(-120:1) -- +(-60:6);

  \draw[valley_opt] (0,0) -- (3,0);
  \draw[mountain] (60:1) -- +(3,0);
  \draw[valley_opt] (120:2) -- +(4,0);
  
  \draw[mountain] (-1,0) -- (-3.3,0);
  \draw[valley_opt] (-1,0) ++(120:1) -- +(-2,0);
  \draw[valley_opt] (-120:1) -- +(-3,0);

  \draw[mountain_opt] (120:2) -- +(-120:1);
  \draw[mountain_opt] (60:1) -- (120:1);
  \draw[valley_opt] (120:1) -- +(-120:1);
  \draw[mountain_opt] (0,0) -- (-120:1);
  \draw[valley_opt] (120:1) -- +(-1,0);
  \draw[mountain_opt] (0,0) -- (-1,0);
  \draw[mountain_opt] (120:1) -- +(60:1);
  \draw[valley_opt] (0,0) -- (60:1);
  \draw[valley_opt] (-120:1) -- (-60:1);
  \draw[mountain_opt] (-60:1) -- (1,0);
  % input/output labels

  \path (0.5,0) +(120:4)  node[rotate=-60] (Fin1) {FALSE};
  \path (-0.5,0) +(120:3.5)  node[rotate=-60] (Tin1) {TRUE};
  \draw[->] (Fin1) -- +(-60:1);
  \draw[->] (Tin1) -- +(-60:1);
  \path (0.5,0) +(120:-2.5)  node[rotate=-60] (Fout1) {FALSE};
  \path (-0.5,0) +(120:-2.5)  node[rotate=-60] (Tout1) {TRUE};
  \draw[->] (Fout1) -- +(-60:1);
  \draw[->] (Tout1) -- +(-60:1);

  \path (-3,0) +(120:0.5) node (Fin2) {FALSE};
  \path (-3.5,0) +(-60:0.5) node (Tin2) {TRUE};
  \draw[->] (Fin2) -- +(1,0);
  \draw[->] (Tin2) -- +(1,0);
  \path (2,0) +(60:1.5) node (Fout2) {FALSE};
  \path (2.5,0) +(60:0.5) node (Tout2) {TRUE};
  \draw[->] (Fout2) -- +(1,0);
  \draw[->] (Tout2) -- +(1,0);
  
  % important points
  \node [circle, fill=black, draw, label=above right:{$\scriptscriptstyle X$}, inner sep=0.13em] at
  (1,0) {};
  \path (0,0) ++(60:1) ++(120:1) node [circle, fill=black, draw, label=above
  right:{$\scriptscriptstyle Y$}, inner sep=0.13em]   {};
  \path (0,0) ++(120:1) ++(-1,0) node [circle, fill=black, draw, label=below
  left:{$\scriptscriptstyle Z$}, inner sep=0.13em]   {};
  \path (-120:1) node [circle, fill=black, draw, label=below
  left:{$\scriptscriptstyle W$}, inner sep=0.13em]   {};
  \path (-1,0) node [circle, fill=black, draw, label=below
  left:{$\scriptscriptstyle Z'$}, inner sep=0.13em]   {};
  \path (-60:1) node [circle, fill=black, draw, label=right:{$\scriptscriptstyle W'$}, inner sep=0.13em]   {};
  \path (120:2) node [circle, fill=black, draw, label=above
  right:{$\scriptscriptstyle Y'$}, inner sep=0.13em]   {};
  \path (60:1) node [circle, fill=black, draw, label=above
  right:{$\scriptscriptstyle X'$}, inner sep=0.13em]   {};
  \path (120:1) node [circle, fill=black, draw, label=below:{$\scriptscriptstyle A$}, inner sep=0.13em]   {};
  \path (0,0) node [circle, fill=black, draw, label=below:{$\scriptscriptstyle B$}, inner sep=0.13em]   {};
\end{tikzpicture}
\captionof{figure}{$\pi/3$ Intersector}
\label{fig:intersect60}
\end{minipage}%
\begin{minipage}{.5\textwidth}
\begin{tikzpicture}[baseline=0, scale=0.85]
\begin{scope}[rotate=60]
  \draw[mountain] (120:4) -- (-60:3);
  \draw[valley_opt] (120:4)++(1,0) -- +(-60:6);
  \draw[valley_opt] (120:4)++(-120:1) -- +(-60:6);

  \draw[valley_opt] (0,0) -- (3,0);
  \draw[mountain] (60:1) -- +(3,0);
  \draw[valley_opt] (120:2) -- +(4,0);
  
  \draw[mountain] (-1,0) -- (-3.3,0);
  \draw[valley_opt] (-1,0) ++(120:1) -- +(-2,0);
  \draw[valley_opt] (-120:1) -- +(-3,0);

  \draw[mountain_opt] (120:2) -- +(-120:1);
  \draw[mountain_opt] (60:1) -- (120:1);
  \draw[valley_opt] (120:1) -- +(-120:1);
  \draw[mountain_opt] (0,0) -- (-120:1);
  \draw[valley_opt] (120:1) -- +(-1,0);
  \draw[mountain_opt] (0,0) -- (-1,0);
  \draw[mountain_opt] (120:1) -- +(60:1);
  \draw[valley_opt] (0,0) -- (60:1);
  \draw[valley_opt] (-120:1) -- (-60:1);
  \draw[mountain_opt] (-60:1) -- (1,0);
  % input/output labels
  \end{scope}

  \path (-0.5,0) +(60:3.5)  node[rotate=60] (Fin1) {FALSE};
  \path (-1.5,0) +(60:4)  node[rotate=60] (Tin1) {TRUE};
  \draw[->] (Fin1) -- +(-120:1);
  \draw[->] (Tin1) -- +(-120:1);
  \path (0.5,0) +(60:-3)  node[rotate=60] (Fout1) {FALSE};
  \path (-0.5,0) +(60:-3)  node[rotate=60] (Tout1) {TRUE};
  \draw[->] (Fout1) -- +(-120:1);
  \draw[->] (Tout1) -- +(-120:1);

  \path (-3,0) +(120:0.5) node (Fin2) {FALSE};
  \path (-3.5,0) +(-60:0.5) node (Tin2) {TRUE};
  \draw[->] (Fin2) -- +(1,0);
  \draw[->] (Tin2) -- +(1,0);
  \path (2,0) +(60:0.5)  node (Fout2) {FALSE};
  \path (2,0) +(-60:0.5) node (Tout2) {TRUE};
  \draw[->] (Fout2) -- +(1,0);
  \draw[->] (Tout2) -- +(1,0);

\end{tikzpicture}
\captionof{figure}{$2\pi/3$ intersector}
\label{fig:intersect120}
\end{minipage}

\end{figure}

\begin{proposition}
  The intersectors (Figures \ref{fig:intersect60} and \ref{fig:intersect120}) work.
\end{proposition}

\begin{proof}
  It suffices to check the claim for the $\pi/3$ intersector; the analysis for
  the $2\pi/3$ intersector is completely analogous.

  By Corollary~\ref{cor:onevalley_int}, exactly one of the valley
  folds in each of the output wires will be active.  

  First, suppose that both inputs are FALSE.  Then at point $Z$ one of the
  angles between active creases is $\pi$, and thus both creases $ZY'$ and $ZA$
  are not active, and $ZZ'$ is active.  Since $ZY'$ is not active, neither is
  $Y'Y$, and since the crease at $2\pi/3$ from $Y$ is not active either, none of
  the creases through $Y$ are active.  Since $YX'$ is not active, $X'A$ must be
  active; thus crease $AZ'$ must be active.  Since $AZ'$ and $ZZ'$ are active,
  crease $Z'B$ must not be active, and crease $Z'W$ must be active.  Since $Z'W$
  and the crease at $\pi$ through $W$ are active, $WB$ must be active and the
  crease at $-\pi/3$ through $W$ must be active.  Thus both outputs are FALSE,
  as desired.

  Now suppose that the left-hand input is TRUE and the top input is FALSE.  Then
  creases $ZY'$ and $ZZ'$ must also be active, as well as crease $Y'Y$.  Since
  $Y'Y$ is active but the crease at $2\pi/3$ from $Y$ is not, the crease at $0$
  from $Y$ must be, and the right-hand output is TRUE.   Since
  $ZZ'$ is active, $Z'B$ must not be, but all other valley folds through $Z'$
  will be.  Since $Z'W$ is active but the crease at $\pi$ from $W$ is not,
  the crease at $-\pi/3$ at $W$ must be; thus the bottom output is FALSE and the
  gadget folds as claimed.  
  
  Now suppose that the left-hand input is FALSE and the top input is TRUE.  Then
  neither the crease at $2\pi/3$ nor the crease at $\pi$ through $Z$ are active,
  and thus none of the creases through $Z$ are active.  Since $ZA$ is not
  active, crease $AY$ cannot be active either; thus crease $YX'$ is active, but
  the crease at $0$ through $Y$ is not active; therefore, the right-hand output
  is FALSE.  Since $YX'$ is active, $X'A$ cannot be active, and thus $X'B$ and
  $X'X$ must be active.  But $AZ'$ is not active, adn thus $Z'B$ must be active,
  while $Z'W$ must not be.  Considering point $W$ we see that since the crease
  at $\pi$ through $W$ must be active, crease $WW'$ must also be active, and
  therefore so is $W'X$.  Since $X'X$ and $W'X$ are active, the crease at $0$
  through $X$ must also be active.  Now there are two possible configurations
  that seem to work: just the crease at $-\pi/6$ through $X$ active (giving the
  desired output), or else creases $XB$, $BW$, and the crease at $-\pi/6$
  through $W$ active---giving the incorrect output.  However, the latter of
  these is exactly the configuration considered in
  Lemma~\ref{lem:problem}, which shows that it cannot flat-fold.  Thus the only
  possible configuration that flat-folds is the correct one.

  Lastly, suppose that both inputs are TRUE.  Since the crease at $\pi$ through
  $Z$ is active but the one at $2\pi/3$ is not, crease $ZA$ must be active, and
  therefore so must $AY$.  Since crease $AY$ is active, three of the valley
  folds through $Y$ must be active, and one of them must be the crease at $0$
  through $Y$; thus the right-hand output is TRUE.  We claim that $YY'$ cannot
  be active.  Indeed, suppose that it is.  Then $Y'Z$ must also be, and
  therefore so is $ZZ'$.  Then $Z'A$ and $Z'W$ must be active, but $Z'B$ not.
  Since $Z'W$ is active but the crease at $\pi$ through $W$ is not, the crease
  at $-\pi/3$ must be, while $WB$ and $WW'$ (and thus $W'X$) are not.  Since
  $Z'B$ and $WB$ are not active, $BX'$ and $BX$ must not be, either.  Since four
  of the creases at $Y$ are active, crease $YX'$ must not be; thus $X'A$ must be
  active, and $X'X$ is not.  Thus none of the creases at $X$ can be active.
  This is exactly the problematic configuration discussed in
  Lemma~\ref{lem:problem}, which cannot flat-fold.  

  Thus $YY'$ is not active.  Then neither are $Y'Z$ or $ZZ'$.  Since $ZZ'$ is
  not active, $Z'B$ must be active and $Z'A$ and $Z'W$ not active.  Since $Z'B$
  is active, $BX'$ must be active and $AX'$ must not be. Thus both $YX'$ and
  $X'X$ are active.  Since the crease at $0$ through $X$ is not active, the
  crease at $-\pi/3$ must be, and the bottom output is TRUE, as desired.
\end{proof}

\begin{figure}[h]
  \centering
  \begin{tikzpicture}
    \draw[line width=4pt,yellow!60!white] (0,0) -- (2,0) (0,0) -- (120:2) (0,0) -- (-120:2);
    \draw[line width=4pt,orange!40!white] (0,0) -- (60:2) (0,0) -- (-2,0) (0,0) -- (-60:2);
    \draw[mountain] (0,0) -- (2,0) (0,0) -- (60:2) (0,0) -- (120:2) (0,0) --
    (-2,0);
    \draw[valley] (0,0) -- (-60:2) (0,0) -- (-120:2);
    \node at (30:1) {$a$};
    \node at (0,1) {$b$};
    \node at (150:1) {$c$};
    \node at (-150:1) {$d$};
    \node at (0,-1) {$e$};
    \node at (-30:1) {$f$};
  \end{tikzpicture}
  \caption{Hexagonal folder}
  \label{fig:hexfold}
\end{figure}

Write $x < y$ to mean that layer $y$ is above layer $x$. 

\begin{lemma} \label{lem:hexfold}
  The only flat-foldsings of Figure~\ref{fig:hexfold} in which region $e$ is
  upward-facing have layer orderings (from top to bottom) $c>d>a>f>e>b$ and
  $a>f>c>d>e>b$.  
\end{lemma}

\begin{proof}
  Since the crease pattern is symmetric with respect to a vertical reflection,
  we can assume without loss of generality that $c > a$.  Now, since $e$ is upward-facing, so is $c$ and $a$, and thus the MV assignment in Figure~\ref{fig:hexfold} implies that 
 any flat-folding of this vertex must have $c > d > e$, $a > f> e$, and $c,a > b$; since we assumed
  $c > a$, this implies that $c > a > b$.  In order to avoid self-intersections,
  we note that after folding all of the orange creases will be lined up, and the
  yellow creases will be lined up.  To figure out possible orderings, we build
  up the layer orderings in stages.

  The first stage is that $c > d > e$ coming from the mountain-valley pattern.
  To build the second stage, consider the possibilities for how $a > f > e$ can
  interleave with this ordering.  Suborderings of the form $a > d > f > e$ and
  $c > f > d > e$ are not allowed, by the taco-taco property (see Section~\ref{ssec:back})
  %\cite[Definition 6.12, Proposition 6.13]{origametry} 
  around the yellow and orange edges,
  respectively. Thus $f$ cannot go between $c$ and $d$; since $a > f$, we must
  therefore have $c > f$, and thus the only possible ordering is
  $c > d > a > f > e$.  For the third stage, consider where $b$ can be inserted
  in this ordering.  Since $a > b$ the only possibilities are
  \[c > d > a > b > f > e \qquad c > d > a > f > b > e \qquad c > d > a > f > e
    > b.\]
  The first of these is forbidden by the taco-taco property around the yellow
  edges; the second is forbidden by the taco-taco property around the orange
  edges.  Thus the only possibility is the last ordering (which is possible by
  first folding the top half behind the bottom half, then folding the right half
  down, and then the left).  
\end{proof}

\begin{figure}
  \begin{minipage}{0.5\textwidth}
\begin{tikzpicture}

  \draw[mountain] (120:4) -- (-60:2);
  \draw[valley] (120:4)++(1,0) -- +(-60:2);
  \draw[valley] (120:2)++(-120:1) -- +(-60:3);

  \draw[mountain] (60:1) -- +(1.5,0);
  \draw[valley] (120:2) -- +(3,0);
  
  \draw[mountain] (-1,0) -- (-3.3,0);
  \draw[valley] (-1,0) ++(120:1) -- +(-2,0);

  \draw[mountain] (120:2) -- +(-120:1);
  \draw[mountain] (60:1) -- (120:1);
  \draw[valley] (120:1) -- +(-120:1);
  \draw[valley] (120:1) -- +(-1,0);
  \draw[mountain] (120:1) -- +(60:1);

%  \path (0.5,0) +(120:4)  node[rotate=-60] (Fin1) {FALSE};
%  \path (-0.5,0) +(120:3.5)  node[rotate=-60] (Tin1) {TRUE};
%  \draw[->] (Fin1) -- +(-60:1);
%  \draw[->] (Tin1) -- +(-60:1);
%  \path (0.5,0) +(120:-2)  node[rotate=-60] (Fout1) {FALSE};
%  \path (-0.5,0) +(120:-2.5)  node[rotate=-60] (Tout1) {TRUE};
%  \draw[->] (Fout1) -- +(-60:1);
%  \draw[->] (Tout1) -- +(-60:1);

%  \path (-3,0) +(120:0.5) node (Fin2) {FALSE};
%  \path (-3.5,0) +(-60:0.5) node (Tin2) {TRUE};
%  \draw[->] (Fin2) -- +(1,0);
%  \draw[->] (Tin2) -- +(1,0);
%  \path (2,0) +(60:1.5) node (Fout2) {FALSE};
%  \path (2.5,0) +(60:0.5) node (Tout2) {TRUE};
%  \draw[->] (Fout2) -- +(1,0);
%  \draw[->] (Tout2) -- +(1,0);

\end{tikzpicture}
\centering
\captionof{figure}{Problematic Intersector configuration}
\label{fig:problem}
\end{minipage}%
  \begin{minipage}{0.5\textwidth}
    \begin{tikzpicture}[scale=1.5]

      \clip (120:1) circle (2.5);

      \draw[color=white, fill=red!20!white] (120:1) +(-1,0) -- +(120:1) --
      +(120:3) arc (120:180:3) -- cycle;
      \draw[color=white, fill=red!20!white] (-1,0) +(0,0) -- +(120:1) -- +(60:1) -- cycle;
      \draw[color=white, fill=red!20!white] (120:1) +(0,0) -- +(120:1) -- +(60:1) -- cycle;
      \draw[color=white, fill=red!20!white] (120:1) ++(60:1) +(0,0) -- +(2,0)
      arc (0:120:2) -- cycle;
      \draw[color=white, fill=red!20!white] (120:1) +(0,0) -- +(3,0) arc
      (0:-60:3) -- cycle;
      \draw[color=white, fill=red!20!white] (-1,0) +(0,0) --+(-60:2) arc
      (-60:-180:2) -- cycle;

      \draw[color=green!40!black!30!white, line width=4] (0,0) -- (120:2) (-1,0) --
      +(60:2) (0,0) ++(60:1) -- ++(-2,0) -- ++(60:1) -- ++(1,0) -- ++(-60:1) --
      ++(-120:1) -- ++(-1,0) -- ++(120:1) -- ++(-120:1) -- ++(-60:1) -- ++(60:1)
      -- ++(-1,0)
      (120:2) ++(-120:1) -- ++(120:1) -- ++(1,0) -- ++(60:1) -- ++(-60:1) --
      ++(60:1) -- ++(-1,0);
      \draw[color=blue!40!black!30!white, line width=4] (60:1) -- (120:1) --
      (120:2) -- ++(60:1) -- ++(1,0) (120:1) -- ++(-120:1) -- ++(-1,0) (120:2)
      --++(-1,0);
      % \draw[color=red!40!black!30!white, line width=4] (0,0) -- ++(120:1) --
      % ++(-1,0) --  ++(120:1) (-2,0) -- +(60:1) (-2,0) -- +(-60:1) (120:1) -- ++
      % (60:1) -- ++(120:1);

  \draw[mountain] (120:4) -- (-60:2);
  \draw[valley] (120:4)++(1,0) -- +(-60:2);
  \draw[valley] (120:2)++(-120:1) -- +(-60:3);

  \draw[mountain] (60:1) -- +(1.5,0);
  \draw[valley] (120:2) -- +(3,0);
  
  \draw[mountain] (-1,0) -- (-3.3,0);
  \draw[valley] (-1,0) ++(120:1) -- +(-2,0);

  \draw[mountain] (120:2) -- +(-120:1);
  \draw[mountain] (60:1) -- (120:1);
  \draw[valley] (120:1) -- +(-120:1);
  \draw[valley] (120:1) -- +(-1,0);
  \draw[mountain] (120:1) -- +(60:1);

%  \path (0.5,0) +(120:4)  node[rotate=-60] (Fin1) {FALSE};
%  \path (-0.5,0) +(120:3.5)  node[rotate=-60] (Tin1) {TRUE};
%  \draw[->] (Fin1) -- +(-60:1);
%  \draw[->] (Tin1) -- +(-60:1);
%  \path (0.5,0) +(120:-2)  node[rotate=-60] (Fout1) {FALSE};
%  \path (-0.5,0) +(120:-2.5)  node[rotate=-60] (Tout1) {TRUE};
%  \draw[->] (Fout1) -- +(-60:1);
%  \draw[->] (Tout1) -- +(-60:1);

%  \path (-3,0) +(120:0.5) node (Fin2) {FALSE};
%  \path (-3.5,0) +(-60:0.5) node (Tin2) {TRUE};
%  \draw[->] (Fin2) -- +(1,0);
%  \draw[->] (Tin2) -- +(1,0);
%  \path (2,0) +(60:1.5) node (Fout2) {FALSE};
%  \path (2.5,0) +(60:0.5) node (Tout2) {TRUE};
%  \draw[->] (Fout2) -- +(1,0);
%  \draw[->] (Tout2) -- +(1,0);

%%%  annotations
  \node[inner sep=2] (a) at (0,{sqrt(3)/3}) {$a$};
  \node[inner sep=2] (b) at (0,{2*sqrt(3)/3}) {$b$};
  \path (0,{sqrt(3)/3}) ++(120:1) node[inner sep=2] (c) {$c$};
  \path (0,{2*sqrt(3)/3}) ++(-1,0) node[inner sep=2] (d) {$d$};
  \path (-1,0) +(0,{sqrt(3)/3}) node[inner sep=2] (e) {$e$};
  \path (-1/2,{sqrt(3)/6}) node[inner sep=2] (f) {$f$};
  \path (0,{sqrt(3)/3}) ++(120:1) ++(60:1) node[inner sep=2] (g) {$g$};
  \path (0,{2*sqrt(3)/3}) ++ (120:1) node[inner sep=2] (h) {$h$};
  \path (0,{sqrt(3)/3}) ++(120:1) ++(-1,0) node[inner sep=2] (i) {$i$};
  \path (-1/2,{sqrt(3)/6}) ++(-1,0) node[inner sep=2] (j) {$j$};
  \path (-1,0) ++(0,{sqrt(3)/3}) ++(-120:1) node[inner sep=2] (k) {$k$};

  \draw[->] (d) -- (i);
  \draw[->] (h) to[bend right] (i);
  \draw[->] (d) -- (c);
  \draw[->] (c) -- (h);
  \draw[->] (j) -- (k);
  \draw[->] (f) -- (a);
  \draw[->] (b) -- (a);
  \draw[->] (b) -- (c);
  \draw[->] (k) -- (f);
  \draw[->] (e) -- (j);
  \draw[->] (e) -- (f);
  \draw[->] (i) to[bend right] (j);
  \draw[->] (g) -- (h);
  \draw[->] (g) -- (b);
  \draw[->] (e) -- (d);

  \draw[->,blue] (b) -- (e);
  \draw[->,red] (c) -- (f);
  
\end{tikzpicture}
\centering
\captionof{figure}{Problematic Intersector configuration with annotations}
\label{fig:problemann}
\end{minipage}
\end{figure}

\begin{lemma} \label{lem:problem}
The problematic configuration (Figure \ref{fig:problem}) will not fold flat.
\end{lemma}
\begin{proof}
  This proof will refer to the annotated version, Figure~\ref{fig:problemann}.
  The colored regions in the diagram are upward-facing; the others are
  downward-facing.  The graph on the nodes shows ordering relations enforced on
  the regions in the graph.  (The green and blue lines show a part of the s-net (defined in Section~\ref{ssec:back})
  %\cite[Definition 6.12]{origametry} 
  consisting of the parts of the s-net which
  overlap region $c$ after folding; a flat-folding must give a well-defined
  ordering on these regions.)  The directions in the graph points from a layer
  that must be lower to one that must be higher.  A cycle in the graph thus
  exhibits a contradiction.

  The black edges in the graph are drawn from local mountain-valley conditions:
  an upward-facing layer must be above a neighboring layer if they differ by a
  mountain fold, and below if they differ by a valley fold.  The blue edge in
  the graph follows from Lemma~\ref{lem:hexfold} applied to regions
  $a$,$b$,$c$,$d$,$e$,$f$.  Note that there is a path from $c$ to $a$; thus if
  we wish to avoid cycles, the layer ordering must have $c < a$ Thus by
  Lemma~\ref{lem:hexfold} the ordering on these layers must be
  $a > f> c > d >e > b$, giving the red arrow.

  Analyzing the black graph with this data, we see that we must have
  \[g < b < e < d < c < h < i < j < k < f < a.\] However, at the top edge of $i$
  there is a taco-tortilla  %\cite[Definition 6.12, Proposition 6.13]{origametry}
  condition  which is violated.  Layer $i$ is between layers $a$ and $b$, but
  does not contain a crease line at the blue line.  The blue line is lined up in
  the folding map with the crease between $a$ and $b$, giving the contradiction.
  Thus the crease pattern cannot flat-fold.
\end{proof}

\subsection{Twists} \label{sec:twist}

A \textit{twist fold} is a crease pattern where $n$ pairs of parallel mountain and valley creases (pleats) meet at an $n$-gon such that folding the pleats flat results in the $n$-gon rotating when folded flat. In addition, standard twist folds require the pleats to fold in the same rotational direction, either all clockwise (with the pleats folding mountain-then-valley in the clockwise direction) or all counterclockwise (valley-then-mountain in the clockwise direction). Twist folds  are ubiquitous in origami tessellations; see, e.g., \cite{Gjerde}. 

Triangle and hexagon twists were pioneered by Fujimoto in the 1970s \cite{Fuj}. Such twists with optional valley creases so as to allow the triangles/hexagons to rotate in either the clockwise or counterclockwise direction are shown in Figures~\ref{fig:hextwist} and \ref{fig:tritwist}. 

We will use these triangle and hexagon twists as gadgets to propagate wire signals in various directions, while also negating them. Their forced rotational nature proves the following Proposition; 
they are simple enough to analyze that we omit the details of the proofs.

\begin{figure}[h]
  \centering
  \begin{minipage}{.5\textwidth}
    \centering
\begin{tikzpicture}
  \draw[mountain] (30:1) -- (0,1) -- (150:1) -- (-150:1) -- (0,-1) -- (-30:1) --
  cycle;
  \draw[mountain] (30:1) -- +(30:2) (0,1) -- +(0,2) (150:1) -- +(150:2) (0,-1)
  -- +(0,-2) (-150:1) -- +(-150:2) (0,-1) -- +(0,-2) (-30:1) -- +(-30:2);
  \draw[valley_opt] (30:1) -- +(0,2) (30:1) -- +(-30:2)
                    (0,1) -- +(30:2) (0,1) -- +(150:2)
                    (150:1) -- +(0,2) (150:1) -- +(-150:2)
                    (-150:1) -- +(150:2) (-150:1) -- +(0,-2)
                    (0,-1) -- +(-150:2) (0,-1) -- +(-30:2)
                    (-30:1) -- +(30:2) (-30:1) -- +(0,-2);
\end{tikzpicture}
\captionof{figure}{Hexagonal twist}
\label{fig:hextwist}
\end{minipage}%
\begin{minipage}{.5\textwidth}
\begin{tikzpicture}[rotate=180]
  \draw[mountain] (30:1) -- (150:1) -- (0,-1) -- cycle;
  \draw[mountain] (30:1) -- +(30:2) (150:1) -- +(150:2) (0,-1) -- +(0,-2);
  \draw[valley_opt] (30:1) -- +(150:3) (30:1) -- +(0,-3)
                    (150:1) -- +(30:3) (150:1) -- +(0,-3)
                    (0,-1) -- +(30:3) (0,-1) -- +(150:3);

\end{tikzpicture}
\captionof{figure}{Triangle twist}
\label{fig:tritwist}
\end{minipage}

\end{figure}

\begin{proposition} \label{prop:twists} The pure hexagonal and triangle twists
  (Figures~\ref{fig:hextwist} and \ref{fig:tritwist}) must flat-fold in
  rotationally-symmetric ways.  In particular, given any designated wire in the
  top half of the gadget as the input, both twists negate and duplicate the
  input value in the other, output wires.
\end{proposition}

%\Tom{Proposition~\ref{prop:twists} needs clarification. Which wires are the inputs and which the outputs?}

%\inna{I've clarified it a bit; what do you think?}

\subsection{The eater} \label{sec:eat}

%\inna{The original version of this papagraph is currently commented out.  It had a comment about being able to ignore noise wires by ``judicious use of  intersector gadgets'', but I'm really not sure this is true.  Intersector  gadgets shift wires off by 1/2 of a wire width, and actually getting things to  line up properly takes a surprising amount of work.  Also, I don't see any  reason why several cells couldn't have misaligned noise which would swamp the  signal without control.  I've edited the paragraph to not need this comment,  but I'm not sure if we should say more.}

% Sometimes twist folds will produce extraneous wires, or \textit{noise} wires
% that are not needed for our construction of Rule 110. Noise wires could be
% simply ignored by judicious use of intersector gadgets, but it makes for
% cleaner crease patterns if we can eliminate noise wires with \textit{eater
% gadgets} that accept any combination of Boolean wire values. A modified
% triangle twist, shown in Figure~\ref{fig:eater} does the job nicely.

Sometimes twist folds will produce extraneous wires, or \textit{noise} wires
that are not needed for our construction of Rule 110. In order to eliminate
these we have \textit{eater gadgets} that accept any combination of Boolean wire
values. In particular, if it can be arranged that all of the noise wires in
adjacent cells either cancel one another out (by colliding in an eater) or
match up, then the cells can be tessellated.  A modified triangle twist, shown in
Figure~\ref{fig:eater} does the job nicely.

\begin{figure}[h] 
\begin{tikzpicture}
% first draw the bidirectional twist
\draw[mountain] (30:1) -- (30:3.5)
                (150:1) -- (150:3.5)
                (0,-1) -- (0,-3.5)
                (30:1) -- (150:1) -- (0,-1) -- cycle;
\draw[valley_opt] (30:1) -- +(0,-4)
                  (30:1) -- +(150:4)
                  (150:1) -- +(30:4)
                  (150:1) -- +(0,-4)
                  (0,-1) -- +(30:4)
                  (0,-1) -- +(150:4);
% then add control folds
\draw[valley_opt] (0,0) -- (30:1)
                  (0,0) -- (150:1)
                  (0,0) -- (0,-1);
\draw[valley_opt] (0,0) -- (0,0.5)
                  (0,0) -- (-150:0.5)
                  (0,0) -- (-30:0.5);

\draw[mountain_opt]   (0,0.5) -- (0,1)
                    (-150:0.5) -- (-150:1)
                    (-30:0.5) -- (-30:1);
%\draw[opt] (150:1) -- (0, 1) -- (30:1) -- (-30:1) -- (0,-1) -- (-150:1) -- cycle;

%input/output labels
% \path (0,0.5) +(150:3) node[rotate=-30] (Fin1) {FALSE};
% \draw[->] (Fin1) -- +(-30:1);
% \path (0,-0.5) +(150:3.5) node[rotate=-30] (Tin1) {TRUE};
% \draw[->] (Tin1) -- +(-30:1);
% \path (0,-0.5) +(30:3.5) node[rotate=30] (Fin2) {FALSE};
% \draw[->] (Fin2) -- +(-150:1);
% \path (0,0.5) +(30:3) node[rotate=30] (Fin2) {TRUE};
% \draw[->] (Fin2) -- +(-150:1);
% \path (-30:0.5) -- +(0,-3) node[rotate=-90] (Fout) {FALSE};
% \path (-150:0.5) -- +(0,-3) node[rotate=-90] (Tout) {TRUE};
% \draw[->] (Fout) -- +(0,-1);
% \draw[->] (Tout) -- +(0,-1);

\end{tikzpicture}
\centering
\caption{Eater}
\label{fig:eater}
\end{figure}

\begin{proposition} \label{prop:eater}
  The values of the three wires entering the eater (Figure~\ref{fig:eater}) are
  independent.  In other words, the eater will flat-fold regardless of the
  values of the three wires.
\end{proposition}

\begin{proof}
%   Can just fold and try it!  It'll always fold flat.
Since the eater crease pattern is rotationally-symmetric as well as reflectively-symmetric about each mountain input axis, all one needs to do to show that the eater will fold flat for any set of inputs is to check when the inputs are all TRUE or have two TRUE and one FALSE inputs.  When all are TRUE the eater turns into a triangle twist and thus can fold flat.  The TRUE, TRUE, FALSE case requires using the short optional mountain crease (and its collinear optional valley) that is between the adjacent optional valleys between one of the TRUE and the FALSE  inputs. It can be readily checked that this, too, is flat-foldable.
\end{proof}

\section{Folding Rule 110}\label{sec:110}

\def\gridsize{60}
\def\topEnd{26}
\def\bottomEnd{-36}

\subsection{The Main Theorem}

Figure~\ref{fig:ori110} shows the schematic of a crease pattern that logically simulates Rule 110. This crease pattern is overlaid on the triangle lattice for clarity; the underlying triangle lattice is not part of the crease pattern. The vertical wires labeled A, B, and C along the top of the Figure are the inputs and the center-most vertical wire at the bottom, labeled OUT, is the output.  The wires and gadgets drawn in color in the Figure control the logical workings of this crease pattern.  The wires and gadgets drawn in grey absorb and direct the noise wires generated by the crease pattern. Also note that the numerous gadgets detailed in the previous Section are simply labeled in Figure~\ref{fig:ori110} by their names, like AND, OR, E (for the eater gadget), and so on rather than drawing all the individual creases. The pale yellow hexagons and triangles are hexagon and triangle twists, respectively.

\begin{theorem}\label{thm:rule110}
If the crease pattern in Figure~\ref{fig:ori110} is given mandatory creases for the input wires A, B, and C, then optional creases can be chosen from the rest to make the crease pattern fold flat, and the result will force truth value of the output wire to follow Rule 110 from the inputs A, B, and C.
\end{theorem}

\begin{proof}
    In the crease pattern, the splitter and crossover gadgets lead the three
  original inputs, or their negatins into a NOR and a pair of NAND clauses and then pass them into triangle twists to produce the following three signals (labeled in Figure~\ref{fig:ori110}):
  \[X =A\vee \neg B \qquad Y =\neg B \wedge C \qquad Z =B\wedge\neg C.\]
  The outputs of these are then led into a NAND and an OR clause to produce the two signals
  \[P = \mathrm{TRUE}\wedge X, Q =  \neg(Y \vee Z).\]

  Finally, the values of $P$ and $Q$ are led to a NAND gadget, so that
  the final output signal is
  \[\neg (P\wedge Q).\] In other words, the output of this crease pattern is:
  \[\neg((A\vee \neg B) \wedge \neg((\neg B\wedge C)\vee (B\wedge \neg C) )).\]
  This simplifies to 
  \begin{equation} \label{eq:1.1}
    (\neg A \wedge B) \vee ((\neg B \wedge C ) \vee (B \wedge \neg C)).
  \end{equation}
  This is exactly what Rule 110 performs. That is, the output is TRUE if
  $B \neq C$ (which makes the second clause in (\ref{eq:1.1}) TRUE), and if
  $B = C$ then this second clause will be FALSE and the output will be the value
  of $(\neg A \wedge B)$. That’s Rule 110.

\end{proof}

We remark that all of the gadgets used in the construction of the flat-folding simulation of Rule 110 generate a unique MV assignment (and thus a unique set of output wire values) for a given set of input wire values. Thus the same is true for our Rule 110 crease pattern (Figure~\ref{fig:ori110})  \textit{except} for the places where two eater gadgets share a wire. Such a shared wire between eater gadgets could be either TRUE or FALSE and not affect the logical constraints of the rest of the crease pattern. Nonetheless, we have demonstrated that \textit{some} flat-folding of the crease pattern will exist to perform Rule 110 computations, which is all that is required to simulate Rule 110.

\subsection{On finiteness, flat-foldable origami tessellations, and Turing
  machines} \label{sec:tess}

The goal of this paper is to prove that falt-folding is Turing-complete,
but what does it mean to make this statement?  A Turing machine is a
finitely-defined object (i.e. a machine with a finite number of states) working
with infinite storage space (the tape) on which is recorded a finite starting
input.  In order to translate this into origami, we form a crease pattern made of of finite-state cells that are tessellated onto the infinite plane.

Consider the following variation of the global flat-foldability problem: We are
given an infinite tessellation of a finite straight-line planar graph $G$ drawn
on our paper $P=\mathbb{R}^2$; the tessellation must have a finite fundamental
region, which is the cell in question. All the edges of $G$ are labeled as
either mandatory mountains, mandatory valleys, optional mountains, or optional
valleys.  $G$ is equipped with a subgraph $G_\mathrm{init}$, which contains all
of the mandatory creases and a subset of the optional creases, and which also
forms a tessellation (i.e. the chosen subgraph is the same in every cell), such
that $G_\mathrm{init}$ is the crease pattern $X_f$ of some isometric folding map
$f:\mathbb{R}^2\to \mathbb{R}^2$ with a global layer ordering $\lambda_f$ whose
MV assignment corresponds to the mountain and valley labeling of
$G_\mathrm{init}$.  This is the ``ground state'' setup of our flat-folding
Turing machine: the blank tape.

To specify a more general problem, we take a graph $G'$ which contains $G$ and
which differs from $G$ by only (a) the addition of finitely many creases, or (b)
the modification of an optional crease into a mandatory crease.  The question
becomes: is there a subset $H$ of $G'$, which which contains all of the
mandatory creases and a subset of the optional creases, such that $H$ is the
crease pattern $X_f$ of some isometric folding map
$f:\mathbb{R}^2\to \mathbb{R}^2$ with a global layer ordering $\lambda_f$ whose
MV assignment corresponds to the mountain and valley labeling of $H$. 

The main result of this paper is that this problem is Turing-complete.
Theorem~\ref{thm:rule110} shows that Rule 110, a finite cellular automaton which
is known to be Turing-complete \cite{Cook}, and in fact P-complete \cite[Theorem
1]{Woods}, can be modeled using a tessellating crease pattern satisfying the
above conditions. The above-mentioned crease pattern $H$ is the output of the
Turing machine from a given input.

The connections between Turing machine computations, Rule 110, and in our flat-foldable crease pattern might seem mysterious to the reader, especially the requirement that such computations be performed with a finite amount of material. Thus, we provide a few details on this in the remainder of this Section.

\begin{definition}
  Consider an elementary cellular automaton with input values encoded as a row
  of cells, colored black for $1$ and white for $0$.  The computation of the
  automaton is recorded as a grid with this input row as the top, and each
  consecutive row determined by the computational rule of the automaton.  This
  gives a coloring of the plane.  A \emph{spaceship} is a self-replicating tile:
  a finite sequence of cells that, if repeated infinitely, will produce the
  original pattern back in a finite number of computations.  
\end{definition}

Computation on Rule 110 is done by observing perturbations in a sequence of
standard spaceships, which are themselves a tessellation $14$ cells wide and $7$
cells tall.  (See \cite{Cook,Woods} for more details.)  Take a $14\times 7$ repetition
of the basic set of cells as the basis of the origami tessellation; since the spaceships
tessellate infinitely, this will fold flat. To compute with Rule 110 a finite
number of perturbations are made to the spaceship pattern; this is encoded in
the crease pattern by setting the direction of input wires to be TRUE or FALSE
by modifying the appropriate optional crease to be a mandatory crease, as
desired, and altering the NAND gate above each input wire to be an eater (by
adding a finite number of optional creases; the extra creases present in the
NAND will not affect the flat-foldability of the gadget as they are all
optional).  This gives the input value to the computation below it.  The part of
the pattern below the given inputs will fold flat with finitely many
modifications to the original spaceship pattern if and only if the original
input reverts to the standard spaceships after finitely many steps.  We conclude that since we need a constant amount of our paper to replicate a Rule 110 spaceship, and by Theorem~\ref{thm:rule110}, we have that our generalized flat-foldable origami
problem is Turing-complete, and in fact P-complete.

\section{Conclusion}\label{sec:concl}

We have shown that folding origami crease patterns with optional creases into a flat state can emulate the behavior of the one-dimensional cellular automaton Rule 110, and can therefore perform the tasks of a universal Turing machine. Actually folding a piece of paper to simulate, say, multiple rows of an instance of Rule 110 using the crease pattern presented here would be a gargantuan task, even for expert origami artists.  Using these methods to perform the computations of a Turing machine using flat origami would be even more arduous, so this is by no means meant to be a practical way to perform computation via origami. 

By way of comparison, we note the existence of prior work from the physics and engineering community on using origami for actual computation, e.g. \cite{Menachem,mechano,printrobots}. These studies build logic gates using \textit{rigid origami}, where a stiff material is folded in a continuous motion so that the creases act like hinges and the regions of material between the creases remain planar, or rigid, during the folding process. Determining whether a crease pattern can be rigidly folded in this way has also been proven to be NP-hard \cite{rigidcom}. While it is likely that rigid origami is also Turing complete as a computational device, to our knowledge no one has proven this. The crease patterns and gadgets in the present work are not rigidly foldable and therefore could not be used as-is in such a proof. Rather, computation performed by flat origami should be viewed discretely, where only the fully flat-folded state provides the desired computational information.

The logic gadgets presented in this paper may be used to simulate other
one-dimensional cellular automata.  For example, a crease pattern to produce
a Sierpinski triangle modulo $2$ would be given by iteration of the cell in
Figure~\ref{fig:sierpinski}; to make this tessellate it is necessary to reflect
consecutive units in each row, and the cells shift half a cell-width in each
row.  The basic cell, as above, is in a dark green box.  To see the Sierpinski
effect, one could color the ``true'' side of the input/output wires blue and the
``false'' red.

\begin{landscape}
  \begin{figure}[h]
    \centering

    % scaling should have x scaled by sqrt(3)/2 more than the y scaling
    \scalebox{0.25}{
    \begin{tikzpicture}[xscale=0.86602540378]
  % base grid
  \clip (-50,\topEnd) rectangle (50,\bottomEnd);

  \foreach \x in {-\gridsize,...,\gridsize} {
    \draw[gridline] (\x,\gridsize) -- (\x,-\gridsize);
  }
  \foreach \y in {\gridsize,...,-\gridsize} {
    \draw[gridline] (0,\y) -- +(\gridsize+2,{(\gridsize+2)/2});
    \draw[gridline] (0,\y) -- +(-\gridsize-2,{(-\gridsize-2)/2});
    \draw[gridline] (0,\y) -- +(\gridsize+2,{(-\gridsize-2)/2});
    \draw[gridline] (0,\y) -- +(-\gridsize-2,{(\gridsize+2)/2});
  }

  %%%%%%% WIRES

  % junk wires

  \jwire[jA]{(-48,16)}{(-72,28)}
  \jwire[jA]{(-48,16)}{(-24,28)}
  \jwire[jB]{(0,16)}{(-24,28)}
  \jwire[jB]{(0,16)}{(24,28)}
  \jwire[jC]{(48,16)}{(24,28)}
  \jwire[jC]{(48,16)}{(72,28)}

  \jwire[jA]{(-48,0)}{(-56,4)}
  \jwire[jA]{(-48,0)}{(-40,4)}
  \jwire[jB]{(-32,0)}{(-40,4)}
  \jwire[jB]{(-32,0)}{(-32,8)}
  \jwire[jB]{(-31,7.5)}{(-31,24.5)}
  \jwire[jB]{(-32,24)}{(-32,\topEnd)}
  \jwire{(-40,4)}{(-40,12)}
  \jwire{(-39,11.5)}{(-39,20.5)} 
  \jwire{(-40,20)}{(-40,\topEnd)}
  \jwire[jB]{(-32,0)}{(-23,-4.5)}
  \jwire[jB]{(-32,-4)}{(-25,-7.5)}
  \jwire[jinfo]{(-44,-11)}{(-48,-13)}
  \jwire[jinfo]{(-52,-11)}{(-48,-13)}
  \jwire{(-48,-13)}{(-48,-15)}
  \jwire[jinfo]{(-40,-15)}{(-32,-19)}
  \jwire{(-48,-15)}{(-44,-17)}
  \jwire{(-48,-15)}{(-52,-17)}
  \jwire{(-45,-17.5)}{(-40,-20)}
  \jwire[info]{(-36,-11)}{(-32,-13)}

  \jwire[jC]{(48,0)}{(56,4)}
  \jwire[jC]{(48,0)}{(40,4)}
  \jwire[jC]{(16,-1)}{(23,-4.5)}
  \jwire[jC]{(16,-5)}{(21,-7.5)}
  \jwire[jB]{(32,0)}{(40,4)}
  \jwire[jB]{(32,0)}{(32,8)}
  \jwire[jB]{(31,7.5)}{(31,24.5)}
  \jwire[jB]{(32,24)}{(32,\topEnd)}
  \jwire{(40,4)}{(40,12)}
  \jwire{(39,11.5)}{(39,20.5)} 
  \jwire{(40,20)}{(40,\topEnd)}
  \jwire[jB]{(32,0)}{(23,-4.5)}
  \jwire[jB]{(32,-4)}{(25,-7.5)}
  \jwire[jinfo]{(44,-11)}{(48,-13)}
  \jwire[jinfo]{(52,-11)}{(48,-13)}
  \jwire{(48,-13)}{(48,-15)}
  \jwire[jinfo]{(40,-15)}{(32,-19)}
  \jwire[jinfo]{(12,-10)}{(16,-12)}
  \jwire[jinfo]{(8,-14)}{(16,-18)}
  \jwire{(48,-15)}{(44,-17)}
  \jwire{(48,-15)}{(52,-17)}
  \jwire{(45,-17.5)}{(40,-20)}
  \jwire{(36,-11)}{(32,-13)}

  \jwire[jA]{(-16,-1)}{(-16,8)}
  \jwire[jA]{(-17,7.5)}{(-17,24.5)}
  \jwire[jA]{(-16,24)}{(-16,\topEnd)}
  \jwire[jA]{(-16,-1)}{(-8,3)}
  \jwire[jB]{(0,-1)}{(-8,3)}
  \jwire{(-8,3)}{(-8,12)}
  \jwire{(-9,11.5)}{(-9,20.5)}
  \jwire{(-8,20)}{(-8,\topEnd)}

  \jwire[jC]{(16,-1)}{(16,8)}
  \jwire[jC]{(17,7.5)}{(17,24.5)}
  \jwire[jC]{(16,24)}{(16,\topEnd)}
  \jwire[jC]{(16,-1)}{(8,3)}
  \jwire[jB]{(0,-1)}{(8,3)}
  \jwire{(8,3)}{(8,12)}
  \jwire{(9,11.5)}{(9,20.5)}
  \jwire{(8,20)}{(8,\topEnd)}

  \jwire[jinfo]{(-4,-10)}{(0,-12)}
  \jwire[jinfo]{(4,-10)}{(0,-12)}
  \jwire{(0,-12)}{(0,-14)}
  \jwire{(0,-14)}{(-4,-16)}
  \jwire{(-3,-16.5)}{(-8,-19)}
  \jwire{(0,-14)}{(4,-16)}
  \jwire{(3,-16.5)}{(8,-19)}

  \jwire[jA]{(-16,-1)}{(-23,-4.5)}
  \jwire[jA]{(-16,-5)}{(-21,-7.5)}
  \jwire{(-12,-10)}{(-16,-12)}
  \jwire[jinfo]{(-8,-14)}{(-16,-18)}

  \jwire[jOUT]{(-30,-18)}{(0,-33)}
  \jwire[jOUT]{(-48,-33)}{(-24,-21)}
  \jwire[jOUT]{(-48,-33)}{(-72,-21)}
  \jwire[jOUT]{(-25,-20.5)}{(-18,-17)}
  \jwire[jOUT]{(30,-18)}{(0,-33)}
  \jwire[jOUT]{(48,-33)}{(24,-21)}
  \jwire[jOUT]{(48,-33)}{(72,-21)}
  \jwire[jOUT]{(25,-20.5)}{(18,-17)}

  \jwire{(-23,-4.5)}{(-23,-6.5)}
  \jwire{(-23,-6.5)}{(-25,-7.5)}
  \jwire{(-23,-6.5)}{(-21,-7.5)}
  \jwire{(-25,-7.5)}{(-25,-9.5)}
  \jwire{(-21,-7.5)}{(-21,-9.5)}
  \jwire{(-21,-9.5)}{(-23,-10.5)}
  \jwire{(-25,-9.5)}{(-23,-10.5)}
  \jwire{(-25,-9.5)}{(-32,-13)}
  \jwire{(-21,-9.5)}{(-16,-12)}
  \jwire{(-23,-10.5)}{(-23,-12.5)}
  \jwire{(-23,-12.5)}{(-30,-16)}
  \jwire{(-23,-12.5)}{(-18,-15)}
  \jwire{(-32,-13)}{(-32,-15)}
  \jwire{(-32,-15)}{(-30,-16)}
  \jwire{(-30,-16)}{(-30,-18)}
  \jwire{(-30,-18)}{(-32,-19)}
  \jwire{(-32,-15)}{(-36,-17)}
  \jwire{(-35,-17.5)}{(-40,-20)}
  \jwire{(-16,-12)}{(-16,-14)}
  \jwire{(-16,-14)}{(-18,-15)}
  \jwire{(-18,-15)}{(-18,-17)}
  \jwire{(-18,-17)}{(-16,-18)}
  \jwire{(-16,-14)}{(-12,-16)}
  \jwire{(-13,-16.5)}{(-8,-19)}
  \jwire{(-40,-20)}{(-40,-29)}
  \jwire{(-39,-28.5)}{(-39,\bottomEnd)}
  \jwire{(-32,-19)}{(-32,-25)}
  \jwire{(-31,-24.5)}{(-31,\bottomEnd)}
  \jwire{(-16,-18)}{(-16,-25)}
  \jwire{(-17,-24.5)}{(-17,\bottomEnd)}
  \jwire{(-8,-19)}{(-8,-29)}
  \jwire{(-9,-28.5)}{(-9,\bottomEnd)}

  \jwire{(23,-4.5)}{(23,-6.5)}
  \jwire{(23,-6.5)}{(25,-7.5)}
  \jwire{(23,-6.5)}{(21,-7.5)}
  \jwire{(25,-7.5)}{(25,-9.5)}
  \jwire{(21,-7.5)}{(21,-9.5)}
  \jwire{(21,-9.5)}{(23,-10.5)}
  \jwire{(25,-9.5)}{(23,-10.5)}
  \jwire{(25,-9.5)}{(32,-13)}
  \jwire{(21,-9.5)}{(16,-12)}
  \jwire{(23,-10.5)}{(23,-12.5)}
  \jwire{(23,-12.5)}{(30,-16)}
  \jwire{(23,-12.5)}{(18,-15)}
  \jwire{(32,-13)}{(32,-15)}
  \jwire{(32,-15)}{(30,-16)}
  \jwire{(30,-16)}{(30,-18)}
  \jwire{(30,-18)}{(32,-19)}
  \jwire{(32,-15)}{(36,-17)}
  \jwire{(35,-17.5)}{(40,-20)}
  \jwire{(16,-12)}{(16,-14)}
  \jwire{(16,-14)}{(18,-15)}
  \jwire{(18,-15)}{(18,-17)}
  \jwire{(18,-17)}{(16,-18)}
  \jwire{(16,-14)}{(12,-16)}
  \jwire{(13,-16.5)}{(8,-19)}
  \jwire{(40,-20)}{(40,-29)}
  \jwire{(39,-28.5)}{(39,\bottomEnd)}
  \jwire{(32,-19)}{(32,-25)}
  \jwire{(31,-24.5)}{(31,\bottomEnd)}
  \jwire{(16,-18)}{(16,-25)}
  \jwire{(17,-24.5)}{(17,\bottomEnd)}
  \jwire{(8,-19)}{(8,-29)}
  \jwire{(9,-28.5)}{(9,\bottomEnd)}

  % information transfer
  
  \wire[notA]{(-48,16)}{(-24,4)}
  \wire[notA]{(-48,16)}{(-72,4)}
  \wire[notA]{(-48,16)}{(-48,0)}
  \wire[A]{(-48,\topEnd)}{(-48,16)}
  \wire[notB]{(0,16)}{(-32,0)}
  \wire[notB]{(0,16)}{(32,0)}
  \wire[notB]{(0,16)}{(0,-1)}
  \wire[B]{(0,\topEnd)}{(0,16)}
  \wire[notC]{(48,16)}{(24,4)}
  \wire[notC]{(48,16)}{(72,4)}
  \wire[notC]{(48,16)}{(48,0)}
  \wire[C]{(48,\topEnd)}{(48,16)}

  \wire[notA]{(-25,3.5)}{(-16,-1)}
  \wire[notC]{(25,3.5)}{(16,-1)}

  \wire[A]{(-48,0)}{(-60,-6)}
  \wire[A]{(-48,0)}{(-48,-5)}
  \wire[A]{(-48,0)}{(-36,-6)}
  \wire[B]{(-32,0)}{(-40,-4)}
  \wire[B]{(-39,-4.5)}{(-44,-7)}
  \wire[B]{(-32,0)}{(-32,-4)}
  \wire[notB]{(-32,-4)}{(-36,-6)}
  \wire[notA]{(-48,-5)}{(-44,-7)}
  \wire[notA]{(-48,-5)}{(-52,-7)}
  \wire[info]{(-44,-7)}{(-44,-11)}
  \wire[info]{(-36,-6)}{(-36,-11)}
  \wire[info]{(-44,-11)}{(-40,-13)}
  \wire[info]{(-36,-11)}{(-40,-13)}
  \wire[info]{(-40,-13)}{(-40,-15)}
  \wire[info]{(-40,-15)}{(-48,-19)}
  \wire[info]{(-58,-15)}{(-48,-19)}
  \wire[OUT]{(-48,-19)}{(-48,-33)}
  \wire[notOUT]{(-48,-33)}{(-48,\bottomEnd)}
  \wire[notOUT]{(0,-33)}{(0,\bottomEnd)}
  \wire[notOUT]{(48,-33)}{(48,\bottomEnd)}
  \wire[notOUT]{(-48,-33)}{(-72,-45)}
  \wire[notOUT]{(-48,-33)}{(-24,-45)}
  \wire[notOUT]{(48,-33)}{(72,-45)}
  \wire[notOUT]{(48,-33)}{(24,-45)}
  \wire[notOUT]{(0,-33)}{(-24,-45)}
  \wire[notOUT]{(0,-33)}{(24,-45)}

  \wire[C]{(48,0)}{(60,-6)}
  \wire[C]{(48,0)}{(48,-5)}
  \wire[C]{(48,0)}{(36,-6)}
  \wire[B]{(32,0)}{(40,-4)}
  \wire[B]{(39,-4.5)}{(44,-7)}
  \wire[B]{(32,0)}{(32,-4)}
  \wire[notB]{(32,-4)}{(36,-6)}
  \wire[notC]{(48,-5)}{(44,-7)}
  \wire[notC]{(48,-5)}{(52,-7)}
  \wire[info]{(44,-7)}{(44,-11)}
  \jwire[junk]{(36,-6)}{(36,-11)}  % used to be actual wire, with info
  \wire[info]{(44,-11)}{(40,-13)}
  % \wire[info]{(36,-11)}{(40,-13)} % set to true
  \draw[mountain] (36,-11) -- (40,-13);
  \draw[valley] (36,-10) -- (40,-12);
  \wire[info]{(40,-13)}{(40,-15)}
  \wire[info]{(40,-15)}{(48,-19)}
  \wire[info]{(58,-15)}{(48,-19)}
  \wire[OUT]{(48,-19)}{(48,-33)}

  \wire[B]{(0,-1)}{(-0,-6)}

  \wire[B]{(0,-1)}{(-12,-7)}
  \wire[A]{(-16,-1)}{(-16,-5)}
  \wire[A]{(-16,-1)}{(-8,-5)}
  \wire[A]{(-9,-5.5)}{(-4,-8)}
  \wire[notB]{(0,-6)}{(-4,-8)}
  \wire[notA]{(-16,-5)}{(-12,-7)}
  \jwire[junk]{(-12,-7)}{(-12,-10)}  %%% used to be actual wire, with info
  \wire[info]{(-4,-8)}{(-4,-10)}
  \wire[info]{(-4,-10)}{(-8,-12)}
  % \wire[info]{(-12,-10)}{(-8,-12)} %%% set to true
  \draw[mountain] (-12,-10) -- (-8,-12);
  \draw[valley] (-12,-9) -- (-8,-11);
  \wire[info]{(-8,-12)}{(-8,-14)}
  \wire[info]{(-8,-14)}{(0,-18)}

  \wire[B]{(0,-1)}{(12,-7)}
  \wire[C]{(16,-1)}{(16,-5)}
  \wire[C]{(16,-1)}{(8,-5)}
  \wire[C]{(9,-5.5)}{(4,-8)}
  \wire[notB]{(0,-6)}{(4,-8)}
  \wire[notC]{(16,-5)}{(12,-7)}
  \wire[info]{(12,-7)}{(12,-10)}
  \wire[info]{(4,-8)}{(4,-10)}
  \wire[info]{(4,-10)}{(8,-12)}
  \wire[info]{(12,-10)}{(8,-12)}
  \wire[info]{(8,-12)}{(8,-14)}
  \wire[info]{(8,-14)}{(0,-18)}

  \wire[OUT]{(0,-18)}{(0,-33)}

  %%%%%% GATES

  % informational gates
  
  \hextwist{(48,16)}
  \hextwist{(0,16)}
  \hextwist{(-48,16)}

  \intersect{(24,4)}{(25,3.5)}
  \intersect{(-24,4)}{(-25,3.5)}

  \hextwist{(-48,0)}
  \hextwist{(-32,0)}
  \hextwist{(-16,-1)}
  \hextwist{(0,-1)}
  \hextwist{(16,-1)}
  \hextwist{(32,0)}
  \hextwist{(48,0)}

  \intersect{(-40,-4)}{(-39,-4.5)}
  \triUp{(-32,-4)}
  \nand{(-36,-6)}
  \nand{(-44,-7)}
  \triUp{(-48,-5)}
  \triUp{(-44,-11)}
  \triUp{(-36,-11)}
  \orgate{(-40,-13)}
  \triUp{(-40,-15)}

  \intersect{(40,-4)}{(39,-4.5)}
  \triUp{(32,-4)}
  \eaterDown{(36,-6)}
  \nor{(44,-7)}
  \triUp{(48,-5)}
  \triUp{(44,-11)}
  \eaterUp{(36,-11)}
  \nand{(40,-13)}
  \triUp{(40,-15)}

  \triUp{(0,-6)}
  \intersect{(-8,-5)}{(-9,-5.5)}
  \nor{(-4,-8)}
  \triUp{(-16,-5)}
  \eaterDown{(-12,-7)}
  \triUp{(-4,-10)}
  \eaterUp{(-12,-10)}
  \nand{(-8,-12)}
  \triUp{(-8,-14)}
  \nand{(0,-18)}
  \intersect{(-3,-16.5)}{(-4,-16)}

  \intersect{(8,-5)}{(9,-5.5)}
  \nand{(4,-8)}
  \triUp{(16,-5)}
  \nand{(12,-7)}
  \triUp{(4,-10)}
  \triUp{(12,-10)}
  \orgate{(8,-12)}
  \triUp{(8,-14)}
  \intersect{(3,-16.5)}{(4,-16)}

  \hextwist{(-48,-33)}
  \hextwist{(0,-33)}
  \hextwist{(48,-33)}
  
  % purely junk gates and intersectors for junk

  \eaterUp{(-40,4)}
  \intersect{(-40,12)}{(-39,11.5)}
  \intersect{(-39,20.5)}{(-40,20)}
  \intersect{(-32,8)}{(-31,7.5)}
  \intersect{(-31,24.5)}{(-32,24)}
  \eaterDown{(-48,-13)}
  \nand{(-48,-19)}
  \eaterUp{(-48,-15)}
  \intersect{(-44,-17)}{(-45,-17.5)}

  \eaterUp{(40,4)}
  \intersect{(40,12)}{(39,11.5)}
  \intersect{(39,20.5)}{(40,20)}
  \intersect{(32,8)}{(31,7.5)}
  \intersect{(31,24.5)}{(32,24)}
  \eaterDown{(48,-13)}
  \nand{(48,-19)}
  \eaterUp{(48,-15)}
  \intersect{(44,-17)}{(45,-17.5)}

  \eaterUp{(-8,3)}
  \intersect{(-8,12)}{(-9,11.5)}
  \intersect{(-9,20.5)}{(-8,20)}
  \intersect{(-16,8)}{(-17,7.5)}
  \intersect{(-17,24.5)}{(-16,24)}

  \eaterUp{(8,3)}
  \intersect{(8,12)}{(9,11.5)}
  \intersect{(9,20.5)}{(8,20)}
  \intersect{(16,8)}{(17,7.5)}
  \intersect{(17,24.5)}{(16,24)}

  \eaterUp{(0,-12)}
  \eaterDown{(0,-14)}
  
  \eaterDown{(-23,-4.5)}
  \eaterUp{(-23,-6.5)}
  \eaterDown{(-21,-7.5)}
  \eaterDown{(-25,-7.5)}
  \eaterUp{(-21,-9.5)}
  \eaterUp{(-25,-9.5)}
  \eaterDown{(-23,-10.5)}
  \eaterUp{(-23,-12.5)}
  \eaterDown{(-40,-20)}
  \eaterDown{(-32,-19)}
  \eaterUp{(-30,-18)}
  \eaterDown{(-30,-16)}
  \eaterUp{(-32,-15)}
  \eaterDown{(-32,-13)}
  \eaterDown{(-16,-12)}
  \eaterUp{(-16,-14)}
  \eaterDown{(-18,-15)}
  \eaterUp{(-18,-17)}
  \eaterDown{(-16,-18)}
  \eaterDown{(-8,-19)}
  \intersect{(-35,-17.5)}{(-36,-17)}
  \intersect{(-13,-16.5)}{(-12,-16)}
  \intersect{(-24,-21)}{(-25,-20.5)}

  \eaterDown{(23,-4.5)}
  \eaterUp{(23,-6.5)}
  \eaterDown{(21,-7.5)}
  \eaterDown{(25,-7.5)}
  \eaterUp{(21,-9.5)}
  \eaterUp{(25,-9.5)}
  \eaterDown{(23,-10.5)}
  \eaterUp{(23,-12.5)}
  \eaterDown{(40,-20)}
  \eaterDown{(32,-19)}
  \eaterUp{(30,-18)}
  \eaterDown{(30,-16)}
  \eaterUp{(32,-15)}
  \eaterDown{(32,-13)}
  \eaterDown{(16,-12)}
  \eaterUp{(16,-14)}
  \eaterDown{(18,-15)}
  \eaterUp{(18,-17)}
  \eaterDown{(16,-18)}
  \eaterDown{(8,-19)}
  \intersect{(35,-17.5)}{(36,-17)}
  \intersect{(13,-16.5)}{(12,-16)}
  \intersect{(24,-21)}{(25,-20.5)}

  \intersect{(-40,-29)}{(-39,-28.5)}
  \intersect{(-32,-25)}{(-31,-24.5)}
  \intersect{(-16,-25)}{(-17,-24.5)}
  \intersect{(-8,-29)}{(-9,-28.5)}
  \intersect{(40,-29)}{(39,-28.5)}
  \intersect{(32,-25)}{(31,-24.5)}
  \intersect{(16,-25)}{(17,-24.5)}
  \intersect{(8,-29)}{(9,-28.5)}

  %%%%%%%%%%%%%% ANNOTATIONS

  \draw[line width=0.5mm, color=green!60!black, dashed] (-48,-33) rectangle (48,16);

  \node[annotation, draw, fill=white] at (-10,-11) {TRUE};
  \node[annotation, draw, fill=white] at (38,-12) {TRUE};
  \node[annotation, draw, fill=white, circle] at (-48,20) {A};
  \node[annotation, draw, fill=white, circle] at (0,20) {B};
  \node[annotation, draw, fill=white, circle] at (48,20) {C};
  \node[annotation, draw, fill=white] at (0,-22) {OUT};
  \node[annotation, draw, fill=white, circle] at (-6,-11) {X};
  \node[annotation, draw, fill=white, circle] at (6,-11) {Y};
  \node[annotation, draw, fill=white, circle] at (10,-11) {Z};
  \node[annotation, draw, fill=white, circle] at (6,-15) {Q};
  \node[annotation, draw, fill=white, circle] at (-6,-15) {P};

    \end{tikzpicture}}

  \caption{Origami crease pattern with optional creases that simulates a Rule 110 cell.}
\label{fig:ori110}
  \end{figure}
\end{landscape}

\begin{figure}[h]
  \centering
  \scalebox{0.4}{
  \begin{tikzpicture}[xscale=0.86602540378]
  % base grid
  \clip (-3.5,2) rectangle (19.5,-22);

  \foreach \x in {-4,-3,...,\gridsize} {
    \draw[gridline] (\x,2) -- (\x,-\gridsize);
  }
  \foreach \y in {\gridsize,...,-\gridsize} {
    \draw[gridline] (-4,\y-1) -- +(\gridsize+2,{(\gridsize+2)/2});
    \draw[gridline] (-4,\y+1) -- +(\gridsize+2,{(-\gridsize-2)/2});
  }

  %%%%%%% WIRES

  % junk wires
  \jwire{(0,-5)}{(6,-2)}
  \jwire{(6,-2)}{(6,5)}
  \jwire{(10,-2)}{(10,5)}
  \jwire{(8,-3)}{(6,-2)}
  \jwire{(8,-3)}{(10,-2)}
  \jwire{(16,-5)}{(10,-2)}
  \jwire{(8,-3)}{(8,-5)}
  \jwire{(8,-5)}{(4,-7)}
  \jwire{(8,-5)}{(12,-7)}
  \jwire{(0,-5)}{(-4,-3)}
  \jwire{(16,-5)}{(20,-3)}
  \jwire{(0,-11)}{(0,-15)}
  \jwire{(16,-11)}{(16,-16)}
  \jwire{(0,-15)}{(2,-16)}
  \jwire{(0,-15)}{(-2,-16)}
  \jwire{(2,-16)}{(4,-15)}
  \jwire{(2,-16)}{(2,-24)}
  \jwire{(-2,-16)}{(-2,-24)}
  \jwire{(16,-16)}{(14,-17)}
  \jwire{(16,-16)}{(18,-17)}
  \jwire{(11,-15.5)}{(14,-17)}
  \jwire{(14,-17)}{(14,-24)}
  \jwire{(18,-17)}{(18,-24)}
  \jwire{(18,-17)}{(20,-16)}
  \jwire{(-2,-16)}{(-4,-15)}

  % info wires
  \wire[notA]{(0,-5)}{(0,-11)}
  \wire[notA]{(0,-5)}{(4,-7)}
  \wire[notA]{(0,-5)}{(-4,-7)}
  \wire[A]{(0,5)}{(0,-5)}
  \wire[notB]{(16,-5)}{(20,-7)}
  \wire[notB]{(16,-5)}{(12,-7)}
  \wire[notB]{(16,-5)}{(16,-11)}
  \wire[B]{(16,5)}{(16,-5)}
  \wire[A]{(4,-7)}{(4,-9)}
  \wire[B]{(12,-7)}{(12,-9)}
  \wire[notA]{(4,-9)}{(0,-11)}
  \wire[notA]{(4,-9)}{(8,-11)}
  \wire[notA]{(-4,-9)}{(0,-11)}
  \wire[notA]{(7,-11.5)}{(11,-13.5)}
  \wire[notB]{(12,-9)}{(16,-11)}
  \wire[notB]{(20,-9)}{(16,-11)}
  \wire[notB]{(12,-9)}{(4,-13)}
  \wire[A]{(0,-11)}{(4,-13)}
  \wire[A]{(0,-11)}{(-4,-13)}
  \wire[B]{(16,-11)}{(11,-13.5)}
  \wire[B]{(16,-11)}{(21,-13.5)}
  \wire[info]{(4,-13)}{(4,-15)}
  \wire[info]{(4,-15)}{(8,-17)}
  \wire[info]{(11,-13.5)}{(11,-15.5)}
  \wire[info]{(11,-15.5)}{(8,-17)}
  \wire[OUT]{(8,-17)}{(8,-24)}

  %%%%%%% GATES

  \hextwist{(0,-5)}
  \hextwist{(0,-11)}
  \hextwist{(16,-5)}
  \hextwist{(16,-11)}
  \triDown{(4,-7)}
  \triUp{(4,-9)}
  \triDown{(12,-7)}
  \triUp{(12,-9)}
  \eaterUp{(8,-5)}
  \eaterDown{(8,-3)}
  \eaterUp{(6,-2)}
  \eaterUp{(10,-2)}
  \intersect{(8,-11)}{(7,-11.5)}
  \nand{(4,-13)}
  \triUp{(4,-15)}
  \eaterDown{(2,-16)}
  \eaterUp{(0,-15)}
  \eaterDown{(-2,-16)}
  \nand{(11,-13.5)}
  \triUp{(11,-15.5)}
  \nor{(8,-17)}
  \eaterDown{(14,-17)}
  \eaterUp{(16,-16)}
  \eaterDown{(18,-17)}

  %%%%%%%%%%%%%% ANNOTATIONS
  %\draw[very thick, color=annotation] (2.5,-11.5) rectangle (5.5,-16.5);
  %\node[annotation, fill=white, draw] at (4,-16.5) {A and !B};
  %\draw[very thick, color=annotation] (9.5,-12) rectangle (12.5,-17);
  %\node[annotation, fill=white, draw] at (11,-17) {!A and B};

  \node[annotation, fill=white, align=center] at (8, -20) {A == B};
  \draw[line width=0.8mm, color=green!60!black, dashed] (0,-1) rectangle (16,-18);
  %\node[annotation, circle, fill=white] at (0,0) {A};
  %\node[annotation, circle, fill=white] at (16,0) {B};

  %\draw[very thick, color=annotation] (12.5,-14.5) rectangle (19.5,-18.5);
  %\node[annotation, fill=white, align=center, draw] at (16,-18.5) {0.5 offset \\ need to reflect \\
  %  consecutive units};

\end{tikzpicture}}
  \caption{A cell in a tessellation that would make a Sierpinski triangle.}
  \label{fig:sierpinski}
\end{figure}

\bibliographystyle{alpha}
\bibliography{HZ}

\end{document}